\documentclass[12pt,amstex]{amsart}
\usepackage{cite}
\usepackage{amsmath,amsthm,amssymb,amsfonts,amscd,verbatim,color}
\usepackage{mathrsfs}
\usepackage{indentfirst}
\usepackage{enumerate}
\usepackage[linktocpage, colorlinks, citecolor=red, anchorcolor=black, linkcolor=red]{hyperref}
\usepackage[top=35mm, bottom=35mm, left=30mm, right=30mm]{geometry}
\usepackage{xcolor}
\usepackage{cleveref} %ÓÃ\CrefÃüÁîÒýÓÃÊ±ÐèÒªµÄºê°ü¡£
\allowdisplaybreaks[4]

\theoremstyle{plain}

\theoremstyle{definition}

\newtheorem{lemma}{Lemma}[section]
\newtheorem{theorem}{Theorem}[section]

\newtheorem{corollary}{Corollary}[section]
\newtheorem{remark}{Remark}[section]

\numberwithin{equation}{section}

\newcommand{\bs}{\boldsymbol}
\newcommand{\mb}{\mathbb}

\newcommand{\mc}{\mathcal}
\newcommand{\ms}{\mathscr}

\def \pa{\partial}

\begin{document}
\title[]{A Birkhoff Normal Form Theorem for Partial Differential Equations on torus}

\let\thefootnote\relax\footnotetext{Supported by NNSFC (Grant Nos. 11822108, 11971299, 12090010, 12090013)}

\author{Jianjun Liu \quad Duohui Xiang}

\address[Jianjun Liu] {School of Mathematics\\ Sichuan University\\ Chengdu 610065, China}
\email{jianjun.liu@scu.edu.cn}

\address[Duohui Xiang] {School of Mathematics\\ Sichuan University\\ Chengdu 610065, China}
\email{duohui.xiang@outlook.com}

\thanks{}

\begin{abstract}
We prove an abstract Birkhoff normal form theorem for Hamiltonian partial differential equations on torus. The normal form is complete up to arbitrary finite order. The proof is based on a valid non-resonant condition and a suitable norm of Hamiltonian function. Then as two examples, we apply this theorem to nonlinear wave equation in one dimension and nonlinear Schr\"{o}dinger equation in high dimension. Consequently, the polynomially long time stability is proved in Sobolev spaces $H^s$ with the index $s$ being much smaller than before. Further, by taking the iterative steps depending on the size of initial datum, we prove sub-exponentially long time stability for these two equations.
\end{abstract}

\maketitle
\tableofcontents

%=============================================================================
\section{Introduction}

In the present paper, we consider long time behavior of small amplitude solutions of Hamiltonian partial differential equations on torus. As an example, let us begin with the nonlinear wave equation
\begin{equation}
\label{form11}
u_{tt}-u_{xx}+mu+f(x,u)=0,\quad x\in\mb{T}:=\mb{R}/2\pi\mb{Z},
\end{equation}
where $m>0$ and the function $f$ is analytic of order at least two with respect to $u$ at the origin, i.e. $f(x,0)=\pa_{u}f(x,u)|_{u=0}=0$.
Denote $\Lambda:=(-\partial_{xx}+m)^{1/2}$ and let $z=\frac{1}{\sqrt{2}}(\Lambda^{\frac{1}{2}}u+{\rm i}\Lambda^{-\frac{1}{2}}u_t)$. Then using the Fourier expansion
$z(t,x)=\sum_{a\in\mb{Z}}z_{a}(t)e^{{\rm i}ax},$
the equation \eqref{form11} is written as
\begin{equation}
\dot{z}_{a}=-{\rm i}\frac{\pa H}{\pa \bar{z}_{a}}
\end{equation}
with the Hamiltonian function
\begin{equation}
\label{form7-30-1}
H=H_0+P=\sum_{a\in\mb{Z}}\omega_{a}|z_{a}|^2+P(z,\bar{z}),
\end{equation}
where $\omega_a=\sqrt{a^{2}+m}$ and $P$ has a zero of order at least three at the origin.

For Hamiltonian of the form \eqref{form7-30-1}, Birkhoff normal form is a fundamental tool to investigate long time dynamics around the origin.
In analogy to classical Birkhoff normal form theorem of finite dimensions, one hopes to find an analytic canonical transformation $\Phi_r$ for $r\geq1$ such that
\begin{equation}
\label{form7-31-2}
H\circ\Phi_r=H_0+\ms{Z}+\ms{R},
\end{equation}
where $\ms{Z}$ is a resonant term of order at most $r+2$, and
the remainder term $\ms{R}$ has a zero of order $r+3$.
%
%For this purpose, in analogy to the non-resonant condition of finite dimensions, one needs that there exists $C>0$ such that the non-trivial combinations of frequencies $\{\omega_a\}_{a\in\mb{Z}}$ satisfy
%\begin{equation}
%|\omega_{a_1}\pm\omega_{a_2}\pm\cdots\pm\omega_{a_{l+2}}|\geq  C
%\end{equation}
%for any $1\leq l\leq r$.
%
For infinite dimensional systems, usually there is no uniformly positive lower bound to ensure non-resonant conditions.
In \cite{BG06}, only these terms with at most two big Fourier modes were eliminated.
As a result, instead of \eqref{form7-31-2}, 
\begin{equation}
H\circ\Phi_r=H_0+\ms{Z}+\ms{R}_{\mathrm{N}}+\ms{R}
\end{equation}
with $\ms{R}_{\mathrm{N}}$ being a truncated remainder in which all terms have at least three big Fourier modes bigger than the truncation parameter $\mathrm{N}$.
Correspondingly, the non-resonant conditions are
\begin{equation}
\label{form7-28-4}
|\omega_{a_1}\pm\omega_{a_2}\pm\cdots\pm\omega_{a_{l+2}}|\geq \frac{\gamma}{\mathrm{N}^{\alpha}},
\end{equation}
where $\gamma>0,\alpha\geq0$, $1\leq l\leq r$, and $|a_1|\leq|a_{2}|\leq\cdots\leq|a_{l+2}|$ with $|a_{l}|\leq \mathrm{N}$.
By tame property, the truncated remainder $\ms{R}_\mathrm{N}$ was well estimated in Sobolev space $H^s$ with the index $s$ sufficiently large. Then an abstract Birkhoff normal form theorem was built to prove long time stability for a wide class of Hamiltonian partial differential equations on torus, such as nonlinear wave equations and nonlinear Schro\"dinger equations.
%proved that these terms have no influence on the stability results.
%
More precisely, if the initial data satisfy $\|u(0)\|_{H^{s}}\leq \varepsilon$, then the solutions satisfy $\|u(t)\|_{H^{s}}\leq2\varepsilon$ for any $|t|\leq \varepsilon^{-r}$, where $s\geq s_{*}$ with
\begin{equation}
s_{*}=O(r^{2}\alpha).
\end{equation}
Specially for the nonlinear wave equation \eqref{form11}, the lower bound $s_{*}=O(r^{8})$.
There are many other papers in which Birkhoff normal form is used to investigate partial differential equations on torus. For nonlinear wave and Schr\"{o}dinger equations, also see \cite{B96,B03,CLY16,BFG20} for example; for equations with unbounded nonlinear vector field, see \cite{D12,YZ14,BD18,Z20} for example; for equations without external parameters, see \cite{B00,BFG20b,BG21} for example.

In the present paper, we eliminate all non-resonant terms of order $r+2$. Thus the Hamiltonian \eqref{form7-30-1} is transformed into \eqref{form7-31-2}, which is a complete Birkhoff normal form up to order $r+2$. 
The non-resonant conditions are 
\begin{equation}
\label{form7-28-3}
|\omega_{a_1}\pm\omega_{a_2}\pm\cdots\pm\omega_{a_{l+2}}|\geq \gamma\Big(\frac{a^*_{1}}{a_2^*\cdots a_{l+2}^*}\Big)^{\tau},
\end{equation}
where $\gamma>0$, $\tau\geq0$, and $a_{1}^{*}\geq a_2^*\geq\dots\geq a_{l+2}^{*}$ denotes the decreasing rearrangement of $\{1+|a_{1}|,1+|a_{2}|,\dots,1+|a_{l+2}|\}$. 
Notice that, without truncation, there is no uniformly positive lower bound for the right hand side of \eqref{form7-28-3}.
Then we build an abstract Birkhoff normal form theorem to prove long time stability for  Hamiltonian partial differential equations on torus, seeing \Cref{th21} and \Cref{co21}. More precisely, the phase space is $H^s$, $s\geq s_{*}$ with 
\begin{equation}
s_{*}=O(r\tau),
\end{equation}
which is usually much smaller than it in \cite{BG06}.
Specially for the nonlinear wave equation \eqref{form11},  we prove $s_{*}=O(r^{4})$, seeing \Cref{th41}.
In addition, for all constants in the Birkhoff normal form process, we concretely calculate the dependence on the iterative steps $r$. Therefore, similarly with \cite{BMP20}, we get longer time of stability by taking big enough $r$ depending on $\varepsilon$, namely
\begin{equation}
T=e^{\frac{1}{2}(\ln\frac{1}{\varepsilon})^{1+\lambda}}\quad\text{for any}\quad 0<\lambda<\frac{1}{4},
\end{equation}
seeing \Cref{co41}.
%%
%Compared with, the difference is that

Actually, \Cref{th21} and \Cref{co21} are valid for partial differential equations on $\mb{T}^{d}$ with $d\geq1$. As an application in higher dimensions, for example, we consider the nonlinear Schr\"{o}dinger equation
\begin{equation}
\label{form12}
{\rm i}\pa_{t}u=-\Delta u+V*u+g(x,u,\bar{u}),\quad x\in\mb{T}^{d}
\end{equation}
with the dimension $d\geq2$, the potential $V$, and the nonlinearity $g$ being smooth with respect to $x$ and analytic of order at least three with respect to $u,\bar{u}$.
In \Cref{th42}, we prove the long time stability in $H^{s}$, where the index $s$ satisfies $s\geq s_{*}=O(r)$ instead of $s_{*}=O(r^{3})$ in \cite{BG06}.
%
%%============
\iffalse
Besides, they showed
$$\sup_{|a|>N^{\frac{\alpha}{\nu}}}|a|^{2s}|\sum_{|b|=|a|}u_b(t)\bar{u}_b(t)-\sum_{|b|=|a|}u_b(0)\bar{u}_b(0)|\leq C\varepsilon^{3},$$
where $\nu>d/2$ is a given number. Roughly speaking, it implies that the energy transfers are allowed only among modes having indexes with equal and large modulus. Instead, we roughly specify $\alpha>\nu$ and thus
$$|a|^{2s}|u_a(t)\bar{u}_a(t)-u_a(0)\bar{u}_a(0)|\leq C\varepsilon^{3},\quad \text{for any}\; a\in\mb{Z}^{d}.$$
%
%The difference is that the  but we ...
\fi
%==================
%
Further, we consider longer time of stability by taking big enough $r$ depending on $\varepsilon$. Previously in \cite{BMP20}, this type of long time stability has been investigated for the equation \eqref{form12} with $d=1$ and an analytic nonlinearity $g(x,|u|^{2})u$. Precisely, the non-resonant conditions are of the form
\begin{equation}
\label{form8-5-1}
|k\cdot\omega|>\gamma\prod_{a\in\mb{Z}}\frac{1}{1+|k_a|^{2}\max\{1,|a|^{\tau}\}}
\end{equation}
with $\gamma>0$, $\tau\geq2$,
%$\omega=\{\omega_a\}_{a\in\mb{Z}}$, $k\in\mb{Z}^{\mb{Z}}$ with $|k|<\infty$,
%
and the stability time is
\begin{equation}
T=e^{\frac{(\ln\frac{1}{\varepsilon})^{2}}{C\ln\ln\frac{1}{\varepsilon}}}
\end{equation}
with $C>0$.
Notice that the non-resonant conditions \eqref{form8-5-1} are invalid for the equation \eqref{form12} with $d\geq2$. 
In the present paper,  for the equation \eqref{form12} with $d\geq2$, we get the stability time
\begin{equation}
\label{form13}
T=\rho^{\frac{(\log_{\rho}\frac{1}{\varepsilon})^{2}}{24\tau\log_{\rho}\log_{\rho}\frac{1}{\varepsilon}}}
\end{equation}
with $\rho=2^{d+4}3^{d}$, seeing \Cref{co42}.
Especially for the equation \eqref{form12} with the nonlinearity $g(u,\bar{u})$,  i.e., not containing the spatial variable $x$ explicitly, we get 
\begin{equation}
T=\rho^{\frac{1}{46}(\log_{\rho}\frac{1}{\varepsilon})^{2}},
\end{equation}
seeing \Cref{re41}.
There are some other papers to investigate longer time of stability, also seeing \cite{B99b,FG13,CLSY18,CLW20,CMW20,CMWZ22} for example.

In order to obtain more concrete estimates in Birkhoff normal form process, we introduce the $s,N$-norm \eqref{form27}. This norm is adapted to partial differential equations with their nonlinearity being smooth with respective to the spatial variable $x$ on tours. 
We emphasize that this norm is well kept under Poisson bracket, seeing the inequailty \eqref{form38} in which the index $s,N$ are miantained.
As a result, we could concretely estimate the remainder term $\ms{R}$ by its Taylor series, seeing \eqref{form217} in \Cref{th21}.
%
%For convenience, consider the norm with the zero momentum condition, i.e. 
When the nonlinearity does not contain the spatial variable $x$ explicitly, the norm \eqref{form27} means that the coefficient of every monomial is  bounded by
\begin{equation}
\label{form8-6-1}
\Big(\frac{a_2^*\cdots a_{l+2}^*}{a_{1}^*}\Big)^{s}.
\end{equation}
%%==========
\iffalse
Preciously, for the homogenous polynomial of order $p+q$ with the form
\begin{equation}
P(z,\bar{z})=\sum_{\substack{\bs{a}\in(\mb{Z}^d)^{p+q}\\a_1+\cdots+a_p-a_{p+1}-\cdots-a_{p+q}=0}}\bs{c}_{\bs{a}}z_{a_1}\cdots z_{a_p}\bar{z}_{a_{p+1}}\cdots\bar{z}_{a_{p+q}},
\end{equation}
%
and then its $s$-norm is 
\begin{equation}
\sup_{\bs{a}}\big|(\frac{a_{1}^*}{a_2^*\cdots a_{p+q}^*})^{s}c_{\bs{a}}\big|
\end{equation}
with $s\geq0$. 
%and $a_{1}^{*}\geq a_2^*\geq\dots\geq a_{p+q}^{*}$ denotes the decreasing rearrangement of $\{1+|k_{1}|,\dots,1+|k_{p}|,1+|h_{1}|,\cdots,1+|h_q|\}$.
%
In nature, the $s$-norm for $s>0$ could be used to estimate the coefficients of monomials in \eqref{form7-28-1} by their Fourier modes. 
\fi
%%==========
%Although
%\begin{equation}
%\frac{a_2^*\cdots a_{l+2}^*}{a_{1}^*}\sim a_3^*\cdots a_{l+2}^*,
%\end{equation}
%
Compared with $a_3^*\cdots a_{l+2}^*$, the structure of $\frac{a_2^*\cdots a_{l+2}^*}{a_{1}^*}$ is better kept under Poisson bracket.
%a multiple $c$ depending on the iterative steps $r$ is required to ensure the iterative process.   
%
This structure has been used for the derivative nonlinear Schr\"odinger equation in \cite{L22}.
In addition, we mention \cite{BDGS07} for long time stability of nonlinear wave equations on Zoll manifold, in which the localized coefficients are of the form
\begin{equation}
\frac{(a_3^*)^{\nu+N}}{(|a_1^*-a_2^*|+a_3^*)^{N}}
\end{equation}
with $\nu\geq0$, $N\in\mb{N}$.
%we also get the estimate of the vector field, seeing \Cref{le32}. 
%
%In the normal form theorem, all the constants in the estimates are concrete, such as the stability radius $R_{*}$ in \eqref{form213} and the coefficient estimates of $\ms{Z}$ and $\ms{R}$. Compared with \cite{BG06}, the \emph{tame modulus} can not be used to estimate its coefficients, and thus they only estimated the Hamiltonian vector field of $\ms{R}$.
%%
%\Cref{th21} and \Cref{co21} are suitable for the bigger Sobolev space $H^{s}$ than it in \cite{BG06}.
%
These are some others papers with coefficient estimates. For nonlinear wave equations on sphere, also see \cite{B08,D15} for example; for nonlinear Schr\"odinger and wave equations with quadratic potential, see \cite{GIP09,Z10} for example.
%other long time stability results  or , see . 
%%

 Now, we lay out an outline of the present paper.
\\\indent In section \ref{sec2}, we firstly introduce some notations and the norm
%of polynomial by using the polynomial weight of the momentum instead of a exponential weight of the momentum in \cite{BMP20}, 
\eqref{form27}. Then we give an abstract Birkhoff normal form theorem with concrete estimates of all constants in the transformation $\phi$, the normal form $\ms{Z}$ and the remainder term $\ms{R}$, seeing Theorem \ref{th21}. Finally, under the non-resonant condition \eqref{form218}, we deduce the long time stability results, seeing Corollary \ref{co21}.
%We show that if the nonlinearity satisfies certain conditions, then there exists a canonical transformation that put the Hamiltonian function $H$ into the form \eqref{form214}. Besides, we give the detailed coefficient estimate by computing explicitly all the constants.
%%
\\\indent In section \ref{sec3}, we prove Theorem \ref{th21} and Corollary \ref{co21}. Firstly, we estimate the vector field, the Poisson bracket, and the solution of homological equation, seeing \Cref{le31}--\Cref{le36}.
%Notice that for monomial $u_{a_1}u_{a_2}u_{a_3}\cdots u_{a_l}$, the coefficient $\frac{|a_1|}{|a_2|\dots|a_{l}|}$ with $|a_1|=\max\{|a_1|,\cdots,|a_l|\}$ has no effect on these estimate.
Then we prove an iterative lemma, seeing Lemma \ref{le37}. In these lemmas, all constants are concretely calculated.
%In the proof of the iterative Lemma, it is a little tedious to adjust the constants in the estimates and the radius of convergence.
Finally, \Cref{th21} and \Cref{co21} are proved.
\\\indent In section \ref{sec4}, we present two applications, i.e., nonlinear wave equation in one dimension and nonlinear Schr\"{o}dinger equation in high dimension. Firstly, for these two equations with nonlinearity being smooth with respect to $x$ and analytic with respect to $u$, we prove polynomially long time stability in $H^s$, seeing \Cref{th41} and \Cref{th42}. Compared with \cite{BG06}, the index $s$ in these two theorems could be smaller.
%
%
%In subsection \ref{sec41}, we consider the nonlinear wave equation \eqref{form11}.
%For convenience, take $m\in[1,2]$ and then we show that for almost all $m$, the solution is sufficiently close to the small initial data for a polynomially long time, seeing Theorem \ref{th41}. We also get a longer time of stability, seeing Corollary \ref{co41}.
%
%In subsection \ref{sec42}, we prove the long time stability of the solutions of the nonlinear Schr\"{o}dinger equation \eqref{form12}, seeing Theorem \ref{th42}. In this case, we can transform the Hamiltonian into a complete Birkhoff normal form up to any finite order, namely the normal form $\ms{Z}$ only depends on the action $I$. Similarly, we have the sub-exponentially long time of stability, seeing \Cref{co42} and \Cref{re41}. 
%
Moreover, the same as \cite{BG06}, we could also assume the nonlinearity being smooth with respective to $u$.
In fact, in \cite{BG06}
%, whether the nonlinearity are smooth or analytic, there is no difference in nature. More precisely, 
the nonlinearity $P$ is written as $P=P'+\ms{R}_{*}$, where the polynomial $P'$ is dealt in Birkhoff normal form and the Taylor remainder $\ms{R}_{*}$ has no influence on the stability.
Further, for these two equations with nonlinearity being analytic with respect to both $x$ and $u$, we prove sub-exponentially long time stability, seeing \Cref{co41} and \Cref{co42}. Furthermore, for nonlinear  Schr\"{o}dinger equation with nonlinearity not containing the spatial variable $x$ explicitly, the stability time in \Cref{co42} could be improved, seeing \Cref{re41}.
%
%We estimate the measure of the external parameters, such as the mass $m$ in the wave equation \eqref{form11} and the potential $V$ in the Schr\"{o}dinger equation \eqref{form12}, seeing Lemma \ref{le42} and Lemma \ref{le44}.
%%

 Appendix contains several lemmas. Lemma \ref{le1} provides one elementary inequality being frequently used in this paper. Lemma \ref{le2} is used to estimate the vector field of homogenous polynomials. Lemma \ref{le3} shows a relationship between Fourier modes.
 %which is used to transform the non-resonant set. 
Lemma \ref{le4}--Lemma \ref{le8} are used to estimate the measure for nonlinear wave equations.
%the measure of a single resonant set $\ms{P}_{\bs{j}}(\gamma,\tau)$ in \eqref{form414} and $\ms{P}_{n,\bs{j}'}(X)$ in \eqref{form026}, respectively, , prove Lemma \ref{le42}
%=============================================================================
\section{Statement of abstract results}
\label{sec2}
Define the Hilbert space $\ell_{s}^{2}(\mb{Z}^{d},\mb{C})$ of the complex sequences $\xi=\{\xi_{a}\}_{a\in\mb{Z}^{d}}$ such that
\begin{equation}
\label{form21}
\|\xi\|^{2}_{s}:=
\sum_{a\in\mb{Z}^{d}}\langle a\rangle^{2s}|\xi_{a}|^{2}<\infty
\end{equation}
with $\langle a\rangle=1+|a|=1+\sqrt{a_{1}^{2}+\cdots+a_{d}^{2}}$.
Notice that for complex function $u(x)=\sum_{a\in\mb{Z}^{d}}\xi_{a}e^{{\rm i}a\cdot x}$ on $\mb{T}^{d}$ with $a\cdot x=a_{1}x_{1}+\cdots+a_{d}x_{d}$, the Sobolev norm $\|u\|_{H^{s}}$ is equivalent to the norm $\|\xi\|_{s}$. The scale of phase spaces
$$\ms{P}_{s}:=\ell_{s}^{2}\oplus\ell_{s}^{2}\ni (\xi,\bar{\xi})= (\{\xi_{a}\}_{a\in\mb{Z}^{d}}, \{\bar{\xi}_{a}\}_{a\in\mb{Z}^{d}})$$ is endowed by the standard symplectic structure $-{\rm i}\sum_{a\in\mb{Z}^{d}}d\xi_{a}\wedge d\bar{\xi}_{a}$.
For a Hamiltonian function $H(\xi,\bar{\xi})$, define its vector field
\begin{equation}
\label{form22}
X_{H}(\xi,\bar{\xi})=-{\rm i}\Big(\frac{\pa H}{\pa\bar{\xi}},
-\frac{\pa H}{\pa \xi}\Big),
\end{equation}
and for two Hamiltonian functions $H(\xi,\bar{\xi})$ and $F(\xi,\bar{\xi})$, define their Poisson bracket
\begin{equation}
\label{form23}
\{H,F\}=-{\rm i}\sum_{a\in\mb{Z}^{d}}\Big(\frac{\pa H}{\pa\xi_a}\frac{\pa F}{\pa\bar{\xi}_a}-\frac{\pa H}{\pa\bar{\xi}_a}\frac{\pa F}{\pa\xi_a}\Big).
\end{equation}
\indent In a brief statement, we identify $\mb{C}^{\mb{Z}^{d}}\times\mb{C}^{\mb{Z}^{d}}\simeq
\mb{C}^{\mb{U}_{2}\times\mb{Z}^{d}}$ with $\mb{U}_{2}=\{\pm1\}$ and use the convenient notation $z=(z_{j})_{j=(\delta,a)\in\mb{U}_{2}\times\mb{Z}^{d}}$, where $$z_{j}=\left\{\begin{aligned}\xi_{a}, &\;\text{when}\;\delta=\;\;\,1,
\\\bar{\xi}_{a}, &\;\text{when}\;\delta=-1.
\end{aligned}\right.$$
Let $\bar{j}=(-\delta,a)$, then $\bar{z}_{j}=z_{\bar{j}}$. Set $\langle j\rangle=\langle a\rangle$ and define
\begin{equation}
\label{form24}
\|z\|_{s}^{2}:=\sum_{j\in\mb{U}_{2}\times\mb{Z}^{d}}\langle j\rangle^{2s}|z_{j}|^{2}=\|\xi\|^{2}_{s}+\|\bar{\xi}\|_{s}^{2}.
\end{equation}
For $\bs{j}=(j_{1},\cdots,j_{l})=(\delta_{k},a_{k})_{k=1}^{l}\in(\mb{U}_{2}\times\mb{Z}^{d})^{l}$, denote the monomial $z_{\bs{j}}=z_{j_{1}}\cdots z_{j_{l}}$ and the momentum
\begin{equation}
\label{form25}
\mc{M}(\bs{j})=\delta_{1}a_{1}+\cdots+\delta_{l}a_{l}.
\end{equation}
%%%
\indent Consider a homogeneous polynomial $f(z)$ of order $l$ with the symmetric form
\begin{equation}
\label{form26}
f(z)=\sum_{\bs{j}\in(\mb{U}_{2}\times\mb{Z}^{d})^{l}}\tilde{f}_{\bs{j}}z_{\bs{j}},
\end{equation}
where for any permutation $\sigma$,  $\tilde{f}_{\sigma(\bs{j})}=\tilde{f}_{\bs{j}}$.
For any $N\geq s\geq0$, define the $s,N$-norm of the homogeneous polynomial
\begin{equation}
\label{form27}
|f|_{s,N}=\sum_{b\in\mb{Z}^{d}}\langle b\rangle^{N-s}\sup_{\mc{M}(\bs{j})=b}\big|(\frac{j_{1}^{*}}{j_{2}^{*}\dots j_{l}^{*}})^{s}\tilde{f}_{\bs{j}}\big|,
\end{equation}
where $j_{1}^{*}\geq\dots\geq j_{l}^{*}$ denotes the decreasing rearrangement of $\{\langle j_{1}\rangle,\dots,\langle j_{l}\rangle\}$. Remark that $j_{1}^{*}\leq\langle\mc{M}({\bs{j}})\rangle j_{2}^{*}\dots j_{l}^{*}$. So $|f|_{s,N}\leq|f|_{s',N}$ for any $s\geq s'$.
Denote $\mc{F}^{s,N}_{l}$ as the class of homogeneous polynomials of order $l$ with the form \eqref{form26} and finite $s,N$-norm.
\\\indent For a function $f(z)=\sum_{i\geq1}f_{i}(z)$ with $f_{i}\in\mc{F}^{s,N}_{i+2}$, define
\begin{equation}
\label{form28}
\langle|f|\rangle_{s,N}^{R}=\sum_{i\geq1}|f_{i}|_{s,N}R^{i}.
\end{equation}
\indent Consider a homogeneous polynomial $\ms{Z}(z)$ of order $l$, which is of the form
$$\ms{Z}(z)=\sum_{\bs{j}\in(\mb{U}_{2}\times\mb{Z}^{d})^{l}}\ms{Z}_{\bs{j}}z_{\bs{j}}.$$
Fix $\tau\geq0$ and $\gamma>0$. The polynomial $\ms{Z}$ is said to be in $(\gamma,\tau)$-normal form with respect to $\omega$ if $\ms{Z}_{\bs{j}}\neq0$ implies
\begin{equation}
\label{form29}
|\delta_{1}\omega_{a_{1}}+\dots+\delta_{l}\omega_{a_{l}}|<\gamma(\frac{j_{1}^{*}}{\langle\mc{M}(\bs{j})\rangle j_{2}^{*}\dots j_{l}^{*}})^{\tau}.
\end{equation}
%%%%%
\begin{theorem}
\label{th21}
Consider the Hamiltonian function
\begin{equation}
\label{form210}
H=H_{0}+P=\sum_{a\in\mb{Z}^{d}}\omega_{a}|\xi_{a}|^{2}+\sum_{i=1}^{+\infty}P_{i},
\end{equation}
where
$P_{i}$ is a homogeneous polynomial of order $i+2$ and there exist $s_{0}\geq0, R'>0$ such that for any $N\geq s_{0}$, there exists $C_{N}>0$ such that for any $R\leq R'$, the higher order perturbation $P$ satisfies
\begin{equation}
\label{form211}
\langle|P|\rangle_{s_{0},N}^{R}\leq C_{N}R.
\end{equation}
Fix $\tau\geq0, \gamma>0$ and a positive integer $r$. Let
\begin{equation}
\label{form212}
\rho=2^{d+4}3^{d}
\end{equation}
and
\begin{equation}
\label{form213}
R_{*}=\min\{\frac{R'}{2},\frac{\gamma}{10\rho r^{4}C_{N}}\}.
\end{equation}
Then for any $s\geq s_{0}+r\tau+\frac{d+1}{2}$ , $N\geq s+\frac{d+1}{2}$ and $R\leq R_{*}$,
there exists a canonical transformation $\phi:B_{s}(\frac{R}{2\rho})\longmapsto B_{s}(\frac{R}{\rho})$ such that
\begin{equation}
\label{form214}
H\circ\phi=(H_{0}+P)\circ\phi=H_{0}+\ms{Z}+\ms{R},
\end{equation}
where
\begin{enumerate}[(i)]
 \item  the transformation satisfies the estimate:
    \begin{equation}
    \label{form215}
    \sup_{\|z\|_{s}\leq R/(2\rho)}\|z-\phi(z)\|_{s}\leq \frac{2C_{N}}{\gamma}R^{2},
    \end{equation}
    and the same estimate is fulfilled by the inverse transformation;
\item $\ms{Z}$ is a polynomial of order at most $r+2$ which is in $(\gamma,\tau)$-normal form with respect to $\omega$ and satisfies the estimate:
\begin{equation}
 \label{form216}
 \langle|\ms{Z}|\rangle_{s_{0}+r\tau,N}^{R}\leq rC_{N}R;
\end{equation}
\item the remainder term $\ms{R}^{(r)}$ is of order at least $r+3$ and satisfies the estimate:
 \begin{equation}
 \label{form217}
  \langle|\ms{R}|\rangle_{s_{0}+r\tau,N}^{R}\leq\frac{C_{N}}{R_{*}^{r}}R^{r+1}.
  \end{equation}
\end{enumerate}
\end{theorem}
%%%%%
Let $\mc{J}_{l}:=\left\{\bs{j}=(\delta_{k},a_{k})_{k=1}^{l}\mid \exists\;\text{permutation}\;\sigma,\;\text{s.t.}\;\forall k,\; \delta_{k}=-\delta_{\sigma(k)},\;|a_{k}|=|a_{\sigma(k)}|\right\}$. We say that the family of frequencies $\{\omega_{a}\}_{a\in\mb{Z}^{d}}$ is non-resonant up to order $r$ if there exist $\tau\geq0,\gamma>0$ such that for any $l\leq r$ and $\bs{j}\in(\mb{U}_{2}\times\mb{Z}^{d})^{l}\backslash\mc{J}_{l}$, one has
\begin{equation}
\label{form218}
|\delta_{1}\omega_{a_{1}}+\dots+\delta_{l}\omega_{a_{l}}|\geq\gamma(\frac{j_{1}^{*}}{\langle\mc{M}(\bs{j})\rangle j_{2}^{*}\dots j_{l}^{*}})^{\tau}.
\end{equation}
%%%%%
For any $a,b\in\mb{Z}^{d}$, denote the super action
$J_{a}:=\sum_{\langle b\rangle=\langle a\rangle}z_{b}\bar{z}_{b}$. If the family of frequencies $\{\omega_{a}\}_{a\in\mb{Z}^{d}}$ is non-resonant, then the normal form $\ms{Z}$ only depends on the super action. Hence, we have the following result of long time stability.
%%%%%
\begin{corollary}
\label{co21}
Consider the Hamiltonian function $H$ in Theorem \ref{th21}. Fix a positive integer $r$ and assume that the family of frequencies $\{\omega_{a}\}_{a\in\mb{Z}^{d}}$ is non-resonant up to order $r+2$ for some $\tau\geq0,\gamma>0$.
For any $s\geq s_{0}+r\tau+\frac{d+1}{2}$,  let $N=s+\frac{d+1}{2}$,
\begin{equation}
\label{form219}
\varepsilon_{*}=\frac{3}{20\rho}R_{*}
\end{equation}
and
\begin{equation}
\label{form220}
c_{1}=\frac{1}{5\rho^{2}C_{N}}(4\varepsilon_{*})^{r}.
\end{equation}
If the initial datum $z(0)\in\ms{P}_{s}$ satisfies
\begin{equation}
\label{form221}
\varepsilon:=\|z(0)\|_{s}\leq\varepsilon_{*},
\end{equation}
then one has
\begin{equation}
\label{form222}
\|z(t)\|_{s}\leq2\varepsilon,\;\text{for}\;|t|\leq c_{1}\varepsilon^{-(r+1)}.
\end{equation}
Moreover, for any $0\leq\nu\leq1$, one has
\begin{equation}
\label{form223}
\sup_{a\in\mb{Z}^{d}}\langle a\rangle^{2s}|J_{a}(t)-J_{a}(0)|\leq c_2\varepsilon^{3-\nu},\;\text{for}\;|t|\leq c_{1}\varepsilon^{-(r+\nu)}
\end{equation}
with $c_2=1+\frac{200C_{N}\rho^{2}}{\gamma}$.
\end{corollary}
%%%%%
\begin{remark}
\label{re21}
Let $\mc{I}_{l}:=\big\{\bs{j}=(\delta_{k},a_{k})_{k=1}^{l}\mid \exists\;\text{permutation}\;\sigma,\;\text{s.t.}\;\forall k,\; \delta_{k}=-\delta_{\sigma(k)},\;a_{k}=a_{\sigma(k)}\big\}$. If
there exist $\tau\geq0,\gamma>0$ such that for any $l\leq r$ and $\bs{j}\in(\mb{U}_{2}\times\mb{Z}^{d})^{l}\backslash\mc{I}_{l}$, the non-resonant conditions \eqref{form218} hold, then the normal form $\ms{Z}$ only depends on the action $I_{a}:=z_{a}\bar{z}_{a}$ and the estimate \eqref{form223} in Corollary \ref{co21} can be replaced by
\begin{equation}
\label{form224}
\sup_{a\in\mb{Z}^{d}}\langle a\rangle^{2s}|I_{a}(t)-I_{a}(0)|\leq c_2\varepsilon^{3-\nu},\;\text{for}\;|t|\leq c_{1}\varepsilon^{-(r+\nu)}.
\end{equation}
\end{remark}
%%%%%

%=========================================================
\section{Proof of normal form theorem and its corollary}
\label{sec3}
Before proving the normal form theorem, we firstly estimate the vector field of  polynomial functions by the following two lemmas.
%%%%%
\begin{lemma}
\label{le31}
Let $f$ be homogenous polynomial of order $r\geq3$ satisfying $|f|_{s_{0},N}<+\infty$ with $N\geq s_{0}+d+1$.  Then for any $s\in[s_{0}+\frac{d+1}{2},N-\frac{d+1}{2}]$, one has
\begin{equation}
\label{form31}
\|X_{f}(z)\|_{s}\leq|f|_{s_{0},N}(\rho\|z\|_{s})^{r-1}
\end{equation}
with the constant $\rho=2^{d+4}3^{d}$.
\end{lemma}
%%%
\begin{proof}
Write
$$f(z)=\sum_{\bs{j}\in(\mb{U}_{2}\times\mb{Z}^{d})^{r}}\tilde{f}_{\bs{j}}z_{\bs{j}}=\sum_{b\in\mb{Z}^{d}}\sum_{\mc{M}(\bs{j})=b}\tilde{f}_{\bs{j}}z_{\bs{j}}$$
and $z'_{j_{i}}=\langle j_{i}\rangle^{s-\frac{d+1}{2}}|z_{j_{i}}|, \forall i=1,\cdots,r$. By the defintion of norm, one has
\begin{align}
\label{form32}
&\|X_{f}(z)\|_{s}\leq\sum_{b\in\mb{Z}^{d}}\Big(\sum_{j_{r}}\langle j_{r}\rangle^{2s}\big|r\sum_{\substack{j_{1},\cdots,j_{r-1}\\\mc{M}(\bs{j})=b}}\tilde{f}_{\bs{j}}z_{j_{1}}\cdots z_{j_{r-1}}\big|^{2}\Big)^{\frac{1}{2}}
\\\notag\leq&r\sum_{b\in\mb{Z}^{d}}\Big(\langle b\rangle^{N-(s-\frac{d+1}{2})}\sup_{\mc{M}(\bs{j})=b}\big|(\frac{j_{1}^{*}}{j_{2}^{*}\dots j_{r}^{*}})^{s-\frac{d+1}{2}}\tilde{f}_{\bs{j}}\big|
\\\notag&\quad\qquad\big(\sum_{j_{r}}\langle j_{r}\rangle^{2s}\big|\sum_{\substack{j_{1},\cdots,j_{r-1}\\\mc{M}(\bs{j})=b}}\langle b\rangle^{s-\frac{d+1}{2}-N}(\frac{j_{2}^{*}\dots j_{r}^{*}}{j_{1}^{*}})^{s-\frac{d+1}{2}}|z_{j_{1}}|\cdots|z_{j_{r-1}}|\big|^{2}\big)^{\frac{1}{2}}\Big)
\\\notag\leq&r|f|_{s-\frac{d+1}{2},N}\sup_{b\in\mb{Z}^{d}}\big(\sum_{j_{r}}\langle j_{r}\rangle^{2s}\big|\sum_{\substack{j_{1},\cdots,j_{r-1}\\\mc{M}(\bs{j})=b}}\langle b\rangle^{s-\frac{d+1}{2}-N}\big(\frac{j_{2}^{*}\dots j_{r}^{*}}{j_{1}^{*}}\big)^{s-\frac{d+1}{2}}|z_{j_{1}}|\cdots|z_{j_{r-1}}|\big|^{2}\big)^{\frac{1}{2}}
\\\notag\leq&r|f|_{s_{0},N}\sup_{b\in\mb{Z}^{d}}\big(\sum_{j_{r}}\langle j_{r}\rangle^{d+1}\big|\sum_{\substack{j_{1},\cdots,j_{r-1}\\\mc{M}(\bs{j})=b}}\langle b\rangle^{s-\frac{d+1}{2}-N}z'_{j_{1}}\cdots z'_{j_{r-1}}\big|^{2}\big)^{\frac{1}{2}}
\\\notag\leq&r|f|_{s_{0},N}\big(\sum_{j_{r}}\langle j_{r}\rangle^{d+1}\big|\sum_{\substack{j_{1},\cdots,j_{r-1},b\\\mc{M}(\bs{j},(-1,b))=0}}\langle b\rangle^{s-\frac{d+1}{2}-N}z'_{j_{1}}\cdots z'_{j_{r-1}}\big|^{2}\big)^{\frac{1}{2}}
\\\notag\leq&r|f|_{s_{0},N}\|y\ast\underbrace{z'\ast\cdots\ast z'}_{r-1}\|_{\frac{d+1}{2}},
\end{align}
where $y=\{y_{j}\}_{j\in\mb{U}_{2}\times\mb{Z}^{d}}$ with $y_{(1,b)}=0$ and $y_{(-1,b)}=\langle b\rangle^{s-\frac{d+1}{2}-N}$.
For any $u=\{u_{j}\}_{j\in\mb{U}_{2}\times\mb{Z}^{d}}$ and $v=\{v_{j}\}_{j\in\mb{U}_{2}\times\mb{Z}^{d}}$, by H\"{o}lder inequality and Lemma \ref{le2} in Appendix, one has
\begin{align*}
\|u\ast v\|_{\frac{d+1}{2}}^{2}=&\sum_{j}\langle j\rangle^{d+1}|\sum_{\substack{j_{1},j_{2}\\\mc{M}(j_{1},j_{2},j)=0}}u_{j_{1}}v_{j_{2}}|^{2}
\\\leq&\sum_{j}\langle j\rangle^{d+1}\big(\sum_{\substack{j_{1},j_{2}\\\mc{M}(j_{1},j_{2},j)=0}}\frac{1}{\langle j_{1}\rangle^{d+1}\langle j_{2}\rangle^{d+1}}\big)\big(\sum_{\substack{j_{1},j_{2}\\\mc{M}(j_{1},j_{2},j)=0}}\langle j_{1}\rangle^{d+1}|u_{j_{1}}|^{2}\langle j_{2}\rangle^{d+1}|v_{j_{2}}|^{2}\big)
\\\leq&\sum_{j}2^{d+3}3^{d}\sum_{\substack{j_{1},j_{2}\\\mc{M}(j_{1},j_{2},j)=0}}\langle j_{1}\rangle^{d+1}|u_{j_{1}}|^{2}\langle j_{2}\rangle^{d+1}|v_{j_{2}}|^{2}
\\=&\rho\|u\|_{\frac{d+1}{2}}^{2}\|v\|_{\frac{d+1}{2}}^{2}.
\end{align*}
By induction, one has
\begin{equation}
\label{form33}
\|y\ast\underbrace{z'\ast\cdots\ast z'}_{r-1}\|_{\frac{d+1}{2}}^{2}\leq\rho^{r-1}\|y\|_{\frac{d+1}{2}}^{2}\|z'\|_{\frac{d+1}{2}}^{2(r-1)}.
\end{equation}
Notice that
\begin{equation}
\label{form34}
\|z'\|_{\frac{d+1}{2}}^{2}=\sum_{j}\langle j\rangle^{d+1}|z'_{j}|^{2}=\sum_{j}\langle j\rangle^{2s}|z_{j}|^{2}=\|z\|_{s}^{2}.
\end{equation}
For any $s\leq N-\frac{d+1}{2}$, by Lemma \ref{le1} in Appendix, one has
\begin{equation}
\label{form35}
\|y\|_{\frac{d+1}{2}}^{2}=\sum_{b\in\mb{Z}^{d}}\langle b\rangle^{d+1}\langle b\rangle^{2s-(d+1)-2N}\leq\sum_{b\in\mb{Z}^{d}}\langle b\rangle^{-(d+1)}\leq3^{d}.
\end{equation}
By \eqref{form32}-\eqref{form35} and the fact that $3^{\frac{d}{2}}r<\rho^{\frac{r-1}{2}}$, one has
$$\|X_{f}(z)\|_{s}\leq 3^{\frac{d}{2}}r|f|_{s_{0},N}\rho^{\frac{r-1}{2}}\|z\|_{s}^{r-1}<|f|_{s_{0},N}(\rho\|z\|_{s})^{r-1}.$$
\end{proof}
%%%%%
\begin{lemma}
\label{le32}
Let $f=\sum_{i\geq1}f_{i}$ with $f_{i}$  a homogenous polynomial of order $i+2$, and $\langle |f|\rangle_{s_{0},N}^{R}<+\infty$ with $N\geq s_{0}+d+1$. Then for any $s\in[s_{0}+\frac{d+1}{2},N-\frac{d+1}{2}]$, one has
\begin{equation}
\label{form36}
\sup_{\|z\|_{s}\leq R/\rho}\|X_{f}(z)\|_{s}\leq R\langle |f|\rangle_{s_{0},N}^{R}
\end{equation}
with the constant $\rho=2^{d+4}3^{d}$.
\end{lemma}
%%%
\begin{proof}
By Lemma \ref{le31}, one has
\begin{equation*}
\sup_{\|z\|_{s}\leq  R/\rho}\|X_{f}(z)\|_{s}
\leq\sum_{i\geq1}\sup_{\|z\|_{s}\leq  R/\rho}\|X_{f_{i}}(z)\|_{s}
\leq\sum_{i\geq1}|f_{i}|_{s_{0},N}R^{i+1}
=R\langle |f|\rangle_{s_{0},N}^{R}.
\end{equation*}
\end{proof}
%%%%%
Consider an analytic function $\chi$ and its corresponding Hamiltonian equation $\dot{z}=X_{\chi}(z)$. Denote by $\Phi_{\chi}^{t}$ the corresponding flow and the time 1 flow $\Phi_{\chi}^{1}$ is called the Lie transformation generated by $\chi$.
Given analytic function $f$, one has
\begin{equation}
\label{form37}
f\circ\Phi_{\chi}^{1}=\sum_{k=0}^{+\infty}\frac{1}{k!}\{\underbrace{\chi,\{\cdots,\{\chi}_{k},f\}\cdots\}\}:=\sum_{k=0}^{+\infty}\frac{1}{k!}ad_{\chi}^{k}f,
\end{equation}
where $ad_{\chi}^{0}f=f, ad_{\chi}^{1}f=\{\chi,f\}$. Next, We will estimate the terms of the series \eqref{form37} in the following lemmas.
%%%%%
\begin{lemma}
\label{le33}
Let $f, g$ be homogenous polynomials of order $r_{1}$ and $r_{2}$, respectively. Assume that $|f|_{s,N}, |g|_{s,N}<+\infty$, then one has
\begin{equation}
\label{form38}
|\{f,g\}|_{s,N}\leq2r_{1}r_{2}|f|_{s,N}|g|_{s,N}.
\end{equation}
\end{lemma}
%%%
\begin{proof}
Write $f,g$ and $\{f,g\}$ the symmetric form explicitly as follows:
$$f(z)=\sum_{\bs{j}'\in(\mb{U}_{2}\times\mb{Z}^{d})^{r_{1}}}\tilde{f}_{\bs{j}'}z_{\bs{j}'}, \qquad
g(z)=\sum_{\bs{j}''\in(\mb{U}_{2}\times\mb{Z}^{d})^{r_{2}}}\tilde{g}_{\bs{j}''}z_{\bs{j}''}$$
and
$$\varphi(z):=\{f,g\}=\sum_{\bs{j}\in(\mb{U}_{2}\times\mb{Z}^{d})^{r_{1}+r_{2}-2}}\tilde{\varphi}_{j_{1},\cdots,j_{r_{1}-1},j_{r_{1}},\cdots,j_{r_{1}+r_{2}-2}}z_{\bs{j}}$$
with $j_{1},\cdots,j_{r_{1}-1}$ of $f$ and $j_{r_{1}},\cdots,j_{r_{1}+r_{2}-2}$ of $g$.
By the definition of Poisson bracket, one has
$$|\tilde{\varphi}_{j_{1},\cdots,j_{r_{1}-1},j_{r_{1}},\cdots,j_{r_{1}+r_{2}-2}}|\leq r_{1}r_{2}|\sum_{j\in\mb{U}_{2}\times\mb{Z}^{d}}\tilde{f}_{j_{1},\cdots,j_{r_{1}-1},j}\tilde{g}_{j_{r_{1}},\cdots,j_{r_{1}+r_{2}-2},\bar{j}}|.$$
So
$$|\varphi|_{s,N}\leq r_{1}r_{2}\sum_{b\in\mb{Z}^{d}}\langle b\rangle^{N-s}\sup_{\mc{M}(\bs{j})=b}|(\frac{j_{1}^{*}}{j_{2}^{*}\dots j_{r_{1}+r_{2}-2}^{*}})^{s}\sum_{j\in\mb{U}_{2}\times\mb{Z}^{d}}\tilde{f}_{j_{1},\cdots,j_{r_{1}-1},j}\tilde{g}_{j_{r_{1}},\cdots,j_{r_{1}+r_{2}-2},\bar{j}}|.$$
Might as well let $\langle j_{1}\rangle\geq\cdots\geq\langle j_{r_{1}-1}\rangle, \langle j_{r_{1}}\rangle\geq\cdots\geq\langle j_{r_{1}+r_{2}-2}\rangle$ and $\langle j_{1}\rangle\geq\langle j_{r_{1}}\rangle$. By $\langle j\rangle=\langle\bar{j}\rangle$, one has
\begin{align*}
\frac{j_{1}^{*}}{j_{2}^{*}\dots j_{r_{1}+r_{2}-2}^{*}}&=
\frac{\langle  j_{1}\rangle}{\langle j_{2}\rangle\dots\langle j_{r_{1}+r_{2}-2}\rangle}
\\&\begin{cases}
&\leq\frac{\langle j\rangle}{\langle j_{1}\rangle\dots\langle j_{r_{1}-1}\rangle}
  \frac{\langle \bar{j}\rangle}{\langle j_{r_{1}}\rangle\dots\langle j_{r_{1}+r_{2}-2}\rangle},
  \;\;\,\qquad\text{when}\ \langle j\rangle\geq\langle j_{1}\rangle\geq\langle j_{r_{1}}\rangle;
\\&=\frac{\langle j_{1}\rangle}{\langle j_{2}\rangle\dots\langle j_{r_{1}-1}\rangle\langle j\rangle}
  \frac{\langle \bar{j}\rangle}{\langle j_{r_{1}}\rangle\dots\langle j_{r_{1}+r_{2}-2}\rangle},
  \quad\quad\text{when}\ \langle j_{1}\rangle\geq\langle j\rangle\geq\langle j_{r_{1}}\rangle;
\\&\leq\frac{\langle j_{1}\rangle}{\langle j_{2}\rangle\dots\langle j_{r_{1}-1}\rangle\langle j\rangle}
  \frac{\langle j_{r_{1}}\rangle}{\langle j_{r_{1}+1}\rangle\dots\langle j_{r_{1}+r_{2}-2}\rangle\langle \bar{j}\rangle},
  \,\;\text{when}\ \langle j_{1}\rangle\geq\langle j_{r_{1}}\rangle\geq\langle j\rangle
  \end{cases}
  \\&=\frac{j_{1}^{'*}}{j_{2}^{'*}\dots j_{r_{1}}^{'*}}\frac{j_{1}^{''*}}{j_{2}^{''*}\dots j_{r_{2}}^{''*}}.
\end{align*}
Notice that $\mc{M}(\bs{j})=\mc{M}(\bs{j}')+\mc{M}(\bs{j}'')$ and $\langle\mc{M}(\bs{j})\rangle\leq\langle\mc{M}(\bs{j}')\rangle\langle\mc{M}(\bs{j}'')\rangle$.
Hence, one has
\begin{align*}
|\varphi|_{s,N}\leq&r_{1}r_{2}\sum_{\substack{b\in\mb{Z}^{d}\\j\in\mb{U}_{2}\times\mb{Z}^{d}}}\sup_{\mc{M}(\bs{j})=b}|\langle b\rangle^{N-s}(\frac{j_{1}^{*}}{j_{2}^{*}\dots j_{r_{1}+r_{2}-2}^{*}})^{s}\tilde{f}_{j_{1},\cdots,j_{r_{1}-1},j}\tilde{g}_{j_{r_{1}},\cdots,j_{r_{1}+r_{2}-2},\bar{j}}|
\\\leq&2r_{1}r_{2}\sum_{b,b'\in\mb{Z}^{d}}\sup_{\substack{\mc{M}(\bs{j})=b\\\mc{M}(\bs{j}')=b'}}|\langle b\rangle^{N-s}(\frac{j_{1}^{*}}{j_{2}^{*}\dots j_{r_{1}+r_{2}-2}^{*}})^{s}\tilde{f}_{j_{1},\cdots,j_{r_{1}-1},j}\tilde{g}_{j_{r_{1}},\cdots,j_{r_{1}+r_{2}-2},\bar{j}}|
\\\leq&2r_{1}r_{2}\sum_{b',b''\in\mb{Z}^{d}}\sup_{\substack{\mc{M}(\bs{j}')=b'\\\mc{M}(\bs{j}'')=b''}}|\langle b'\rangle^{N-s}\langle b''\rangle^{N-s}(\frac{j_{1}^{'*}}{j_{2}^{'*}\dots j_{r_{1}}^{'*}})^{s}(\frac{j_{1}^{''*}}{j_{2}^{''*}\dots j_{r_{2}}^{''*}})^{s}\tilde{f}_{\bs{j}'}\tilde{g}_{\bs{j}''}|
\\=&2r_{1}r_{2}|f|_{s,N}|g|_{s,N}.
\end{align*}
\end{proof}
%%%%%
\begin{lemma}
\label{le34}
Let $\chi$ be a homogenous polynomial of order $r\geq3$ and $f=\sum_{i\geq1}f_{i}$ with $f_{i}$ a homogenous polynomial of order $i+2$. Assume that $|\chi|_{s,N}, \langle |f|\rangle_{s,N}^{R}<+\infty$, then for any $\theta\in(0,\frac{1}{2}]$, one has
\begin{equation}
\label{form39}
\langle|\{\chi,f\}|\rangle_{s,N}^{R(1-\theta)}\leq\frac{4r}{e\theta}\langle|f|\rangle_{s,N}^{R}|\chi|_{s,N}R^{r-2}.
\end{equation}
\end{lemma}
%%%
\begin{proof}
By Lemma \ref{le33}, one has
\begin{align*}
\langle|\{\chi,f\}|\rangle_{s,N}^{R(1-\theta)}&=\sum_{i\geq1}|\{\chi,f_{i}\}|_{s,N}R^{i+r-2}(1-\theta)^{i+r-2}
\\&\leq\sum_{i\geq1}2r(i+2)|\chi|_{s,N}R^{r-2}|f_{i}|_{s,N}R^{i}(1-\theta)^{i+r-2}
\\&\leq2r(1-\theta)^{r-4}\sup_{i\geq1}(i+2)(1-\theta)^{i+2}\langle|f|\rangle_{s,N}^{R}|\chi|_{s,N}R^{r-2}
\\&\leq\frac{2r(1-\theta)^{r-4}}{e\theta}\langle|f|\rangle_{s,N}^{R}|\chi|_{s,N}R^{r-2}
\\&\leq\frac{4r}{e\theta}\langle|f|\rangle_{s,N}^{R}|\chi|_{s,N}R^{r-2}.
\end{align*}
\end{proof}
%%%%%
\begin{lemma}
\label{le35}
Let $\chi$ be a homogenous polynomial of order $r\geq3$ and $f=\sum_{i\geq1}f_{i}$ with $f_{i}$ a homogenous polynomial of order $i+2$. Assume that $|\chi|_{s,N}, \langle |f|\rangle_{s,N}^{R}<+\infty$, then for any $\theta\in(0,\frac{1}{2}]$, one has
\begin{equation}
\label{form310}
\langle|\frac{1}{k!}ad_{\chi}^{k}f|\rangle_{s,N}^{R(1-\theta)}\leq(\frac{4r}{\theta}|\chi|_{s,N}R^{r-2})^{k}\langle|f|\rangle_{s,N}^{R}, \quad\forall k\geq1.
\end{equation}
\end{lemma}
%%%
\begin{proof}
Fix $k$ and denote $\theta'=\theta/k$. For any $l\leq k$, let
$C_{l}^{(k)}=\langle|\frac{1}{l!}ad_{\chi}^{l}f|\rangle_{s,N}^{R(1-\theta'l)}$. By Lemma \ref{le34}, one has
\begin{align*}
C_{l}^{(k)}&=\frac{1}{l}\langle|\{\chi,\frac{1}{(l-1)!}ad_{\chi}^{l-1}f\}|\rangle_{s,N}^{R(1-\theta'l)}
\\&\leq\frac{1}{l}\frac{4r}{e\theta'}\langle|\frac{1}{(l-1)!}ad_{\chi}^{l-1}f|\rangle_{s,N}^{R(1-\theta'(l-1))}|\chi|_{s,N}R^{r-2}(1-\theta'(l-1))^{r-2}
\\&\leq\frac{1}{l}\frac{4r}{e\theta'}|\chi|_{s,N}R^{r-2}C_{l-1}^{(k)}
\end{align*}
Notice that $C_{0}^{(k)}=\langle|f|\rangle_{s,N}^{R}$. By induction and the inequality $k^{k}\leq k!e^{k}$, one has
\begin{align*}
C_{k}^{(k)}&\leq\frac{1}{k!}(\frac{4r}{e\theta'}|\chi|_{s,N}R^{r-2})^{k}C_{0}^{(k)}
\\&=\frac{1}{k!}(\frac{4rk}{e\theta}|\chi|_{s,N}R^{r-2})^{k}\langle|f|\rangle_{s,N}^{R}
\\&\leq(\frac{4r}{\theta}|\chi|_{s,N}R^{r-2})^{k}\langle|f|\rangle_{s,N}^{R}.
\end{align*}
Then \eqref{form310} holds.
\end{proof}
%%%%%
In order to construct the canonical transformation $\phi$ in Theorem \ref{th21}, the following lemma solves homological equations.
%%%%%
\begin{lemma}
\label{le36}
Let $f$ be homogenous polynomial of order $r\geq3$ and $H_{0}=\sum_{a\in\mb{Z}^{d}}\omega_{a}|\xi_{a}|^{2}$. If $|f|_{s,N}<+\infty$, then there exist homogenous polynomials $\chi,\ms{Z}$ of order $r$ such that
\begin{equation}
\label{form311}
\{H_{0},\chi\}+\ms{Z}=f,
\end{equation}
where  $\ms{Z}$ is in $(\gamma,\tau)$-normal form with respect to $\omega$. Moverover, $\chi$ and $\ms{Z}$ fulfill the estimates
\begin{equation}
\label{form312}
|\chi|_{s+\tau,N}\leq\frac{1}{\gamma}|f|_{s,N}\quad\text{and}\quad|\ms{Z}|_{s,N}\leq|f|_{s,N}.
\end{equation}
\end{lemma}
%%%
\begin{proof}
Expand $f,\chi,\ms{Z}$ in Taylor series, namely,
$$f(z)=\sum_{\bs{j}\in(\mb{U}_{2}\times\mb{Z}^{d})^{r}}f_{\bs{j}}z_{\bs{j}},\quad
\chi(z)=\sum_{\bs{j}\in(\mb{U}_{2}\times\mb{Z}^{d})^{r}}\chi_{\bs{j}}z_{\bs{j}},\quad
\ms{Z}(z)=\sum_{\bs{j}\in(\mb{U}_{2}\times\mb{Z}^{d})^{r}}\ms{Z}_{\bs{j}}z_{\bs{j}},$$
and then for any $\bs{j}=(\delta_{k},a_{k})_{k=1}^{r}$, the equation \eqref{form311} is equivalent to
$$-{\rm i}(\delta_{1}\omega_{a_{1}}+\dots+\delta_{r}\omega_{a_{r}})\chi_{\bs{j}}+\ms{Z}_{\bs{j}}=f_{\bs{j}}.$$
Define
\begin{align*}
\chi_{\bs{j}}=&
\begin{cases}
\frac{f_{\bs{j}}}{-{\rm i}(\delta_{1}\omega_{a_{1}}+\dots+\delta_{r}\omega_{a_{r}})},& \text{when}\;|\delta_{1}\omega_{a_{1}}+\dots+\delta_{r}\omega_{a_{r}}|\geq\gamma(\frac{j_{1}^{*}}{\langle\mc{M}(\bs{j})\rangle j_{2}^{*}\dots j_{r}^{*}})^{\tau}, \\
0,& \text{otherwise};
\end{cases}
\\\ms{Z}_{\bs{j}}=&
\begin{cases}
f_{\bs{j}},& \qquad\text{when}\;|\delta_{1}\omega_{a_{1}}+\dots+\delta_{r}\omega_{a_{r}}|<\gamma (\frac{j_{1}^{*}}{\langle\mc{M}(\bs{j})\rangle j_{2}^{*}\dots j_{r}^{*}})^{\tau}, \\
0,& \qquad\text{otherwise}.
\end{cases}
\end{align*}
By construction, $\ms{Z}$ is in $(\gamma,\tau)$-normal form with respect to $\omega$ and the estimates \eqref{form312} immediately follow from the definition of the norm.
\end{proof}
%%%%%
In order to state the iterative lemma, we need some notations. Fix {\color{blue} an} integer $r\geq1$ and for any positive integer $l\leq r$, write $R_{*l}=R_{*}(2-\frac{l}{r})$; similarly, for any $R\leq R_{*}$, write $R_{l}=R(2-\frac{l}{r})$. We proceed by induction from $l-1$ to $l$. Initially, $H^{(0)}=H_{0}+\ms{Z}^{(0)}+\ms{R}^{(0)}$ with $\ms{Z}^{(0)}=0$ and $\ms{R}^{(0)}=P$.
%%%%%
\begin{lemma}[Iterative Lemma]
\label{le37}
Consider the Hamiltonian function \eqref{form210} in Theorem \ref{th21}. There exists a homogenous polynomial Hamiltonian $\chi_{l}$ of order $l+2$
such that for any  $s\geq s_{0}+l\tau+\frac{d+1}{2}$ and $N\geq s+\frac{d+1}{2}$, the Lie transformation $\Phi_{\chi_{l}}^{1}:B_{s}(\frac{R_{l}}{2\rho})\mapsto B_{s}(\frac{R_{l-1}}{2\rho})$ puts $H$ in  normal form:
\begin{equation}
\label{form313}
H^{(l)}:=H^{(l-1)}\circ\Phi_{\chi_{l}}^{1}=H_{0}+\ms{Z}^{(l)}+\ms{R}^{(l)}{\color{blue} ,}
\end{equation}
where
\begin{enumerate}[(i)]
\item the transformation satisfies the estimate:
\begin{equation}
\label{form314}
 \sup_{\|z\|_{s}\leq R_{l}/(2\rho)}\|z-\Phi^{1}_{\chi_{l}}(z)\|_{s}\leq\frac{C_{N}}{2^{l+1}\gamma}R_{l-1}^{2};
\end{equation}
  \item $\ms{Z}^{(l)}$ is a polynomial of order at most $l+2$ which is in $(\gamma,\tau)$-normal form with respect to $\omega$ and
  \begin{equation}
  \label{form315}
  \langle|\ms{Z}^{(l)}|\rangle_{s_{0}+l\tau,N}^{R_{l}}\leq lC_{N}R_{l};
  \end{equation}
  \item the remainder term $\ms{R}^{(l)}$ is of order at least $l+3$ and
  \begin{equation}
  \label{form316}
  \langle|\ms{R}^{(l)}|\rangle_{s_{0}+l\tau,N}^{R_{l}}\leq C_{N}R_{l}(\frac{R_{l}}{R_{*l}})^{l}.
  \end{equation}
\end{enumerate}
\end{lemma}
%%%
\begin{proof}
Let $\ms{R}^{(l-1)}=\sum_{i\geq l}\ms{R}^{(l-1)}_{i}$ with $\ms{R}^{(l-1)}_{i}$ a homogenous polynomial of order $i+2$. By Lemma \ref{le36}, there exist homogenous polynomials $\chi_{l}$ and $\ms{Z}_{l}$ of order $l+2$, where $\ms{Z}_{l}$ is in $(\gamma,\tau)$-normal form with respect to $\omega$, such that
$$\{H_{0},\chi_{l}\}+\ms{Z}_{l}=\ms{R}^{(l-1)}_{l}$$
 with the following estimates
\begin{align}
\label{form317}
&|\chi_{l}|_{s_{0}+l\tau,N}\leq\frac{1}{\gamma}|\ms{R}^{(l-1)}_{l}|_{s_{0}+(l-1)\tau,N},
\\\label{form318}&|\ms{Z}_{l}|_{s_{0}+(l-1)\tau,N}\leq|\ms{R}^{(l-1)}_{l}|_{s_{0}+(l-1)\tau,N}.
\end{align}
By the formula \eqref{form37}, one has
\begin{align*}
H^{(l-1)}\circ\Phi_{\chi_{l}}^{1}=&H_{0}+\{\chi_{l},H_{0}\}+\ms{Z}^{(l-1)}+\ms{R}^{(l-1)}
\\&+\sum_{i=2}^{+\infty}\frac{1}{i!}ad_{\chi_{l}}^{i}H_{0}+\sum_{i=1}^{+\infty}\frac{1}{i!}ad_{\chi_{l}}^{i}\ms{Z}^{(l-1)}+\sum_{i=1}^{+\infty}\frac{1}{i!}ad_{\chi_{l}}^{i}\ms{R}^{(l-1)}
\\:=&H_{0}+\ms{Z}^{(l)}+\ms{R}^{(l)},
\end{align*}
where
\begin{align}
\label{form319}&\ms{Z}^{(l)}=\ms{Z}^{(l-1)}+\ms{Z}_{l},
\\\label{form320}&\ms{R}^{(l)}=\ms{R}^{(l-1)}-\ms{R}^{(l-1)}_{l}+\sum_{i=2}^{+\infty}\frac{1}{i!}ad_{\chi_{l}}^{i}H_{0}+\sum_{i=1}^{+\infty}\frac{1}{i!}ad_{\chi_{l}}^{i}\ms{Z}^{(l-1)}+\sum_{i=1}^{+\infty}\frac{1}{i!}ad_{\chi_{l}}^{i}\ms{R}^{(l-1)}.
\end{align}
\indent  By \eqref{form316} for $l-1$, one has
\begin{equation}
\label{form321}
|\ms{R}^{(l-1)}_{l}|_{s_{0}+(l-1)\tau,N}\leq\frac{C_{N}}{R_{*l-1}^{l-1}}.
\end{equation}
By \eqref{form317} and \eqref{form321}, one has
\begin{equation}
\label{form322}
|\chi_{l}|_{s_{0}+l\tau,N}\leq\frac{C_{N}}{\gamma R^{l-1}_{*l-1}}.
\end{equation}
By Lemma \ref{le31} and \eqref{form322}, one has
$$\sup_{\|z\|_{s}\leq R_{l-1}/(2\rho)}\|X_{\chi_{l}}(z)\|_{s}\leq|\chi_{l}|_{s_{0}+l\tau,N}(\frac{R_{l-1}}{2})^{l+1}
\leq\frac{C_{N}}{2^{l+1}\gamma}(\frac{R_{l-1}}{R_{*l-1}})^{l-1}R_{l-1}^{2}
\leq\frac{C_{N}}{2^{l+1}\gamma}R_{l-1}^{2}.$$
Notice that $R_{l-1}\leq2R$  by the definition of $R_{l-1}$ and $R\leq R_{*}\leq\frac{\gamma}{2\rho rC_{N}}$ by \eqref{form213}. Then one has
$$\sup_{\|z\|_{s}\leq R_{l-1}/(2\rho)}\|X_{\chi_{l}}(z)\|_{s}\leq\frac{C_{N}}{2^{l+1}\gamma}(2R)^{2}
\leq\frac{C_{N}}{2^{l+1}\gamma}4R_{*}R\leq\frac{R}{2\rho r}
=\frac{R_{l-1}-R_{l}}{2\rho}.$$
So the transformation $\Phi_{\chi_{l}}^{1}$ is well defined in $B_{s}(\frac{R_{l}}{2\rho})$ with the estimate
$$\sup_{\|z\|_{s}\leq R_{l}/(2\rho)}\|z-\Phi_{\chi_{l}}^{1}(z)\|_{s}\leq\sup_{\|z\|_{s}\leq R_{l-1}/(2\rho)}\|X_{\chi_{l}}(z)\|_{s}\leq\frac{C_{N}}{2^{l+1}\gamma}R_{l-1}^{2}.$$
\indent Next, we will estimate $\langle|\ms{Z}^{(l)}|\rangle_{s_{0}+l\tau,N}^{R_{l}}$. By \eqref{form318} and \eqref{form321}, one has
\begin{equation}
\label{form323}
|\ms{Z}_{l}|_{s_{0}+l\tau,N}R_{l}^{l}\leq|\ms{Z}_{l}|_{s_{0}+(l-1)\tau,N}R_{l}^{l}\leq C_{N}R_{l}(\frac{R_{l}}{R_{*l-1}})^{l-1}
\leq C_{N}R_{l}.
\end{equation}
By \eqref{form315} for $l-1$, one has
\begin{equation}
\label{form324}
\langle|\ms{Z}^{(l-1)}|\rangle_{s_{0}+l\tau,N}^{R_{l}}\leq\langle|\ms{Z}^{(l-1)}|\rangle_{s_{0}+(l-1)\tau,N}^{R_{l-1}}\frac{R_{l}}{R_{l-1}}\leq(l-1)C_{N}R_{l}.
\end{equation}
In view of \eqref{form319}, by \eqref{form323} and \eqref{form324}, one has
\begin{align*}
\langle|\ms{Z}^{(l)}|\rangle_{s_{0}+l\tau,N}^{R_{l}}=\langle|\ms{Z}^{(l-1)}|\rangle_{s_{0}+l\tau,N}^{R_{l}}+|\ms{Z}_{l}|_{s_{0}+l\tau,N}R_{l}^{l}
\leq lC_{N}R_{l}.
\end{align*}
\indent Finally, we estimate $\langle|\ms{R}^{(l)}|\rangle_{s_{0}+l\tau,N}^{R_{l}}$ by dividing $\ms{R}^{(l)}$ in \eqref{form320} into four parts .
\\\indent1) Taking $R_{l-1}=R_{*l-1}$ in \eqref{form316} for $l-1$, one has
\begin{align*}
\langle|\ms{R}^{(l-1)}-\ms{R}^{(l-1)}_{l}|\rangle_{s_{0}+(l-1)\tau,N}^{R_{*l-1}}
\leq\langle|\ms{R}^{(l-1)}|\rangle_{s_{0}+(l-1)\tau,N}^{R_{*l-1}}\leq C_{N}R_{*l-1},
\end{align*}
and hence
\begin{align}
\label{form325}
\langle|\ms{R}^{(l-1)}-\ms{R}^{(l-1)}_{l}|\rangle_{s_{0}+l\tau,N}^{R_{l}}
&\leq\langle|\ms{R}^{(l-1)}-\ms{R}^{(l-1)}_{l}|\rangle_{s_{0}+(l-1)\tau,N}^{R_{*l-1}}(\frac{R_{l}}{R_{*l-1}})^{l+1}
\\\notag&\leq C_{N}R_{*l-1}(\frac{R_{l}}{R_{*l-1}})^{l+1}
\\\notag&= (1-\frac{1}{2r-l+1})^{l}C_{N}R_{l}(\frac{R_{l}}{R_{*l}})^{l}
\\\notag&\leq (1-\frac{1}{2r})C_{N}R_{l}(\frac{R_{l}}{R_{*l}})^{l}.
\end{align}
\\\indent2) Notice that $ad_{\chi_{l}}^{1}H_0=\ms{Z}_{l}-\ms{R}^{(l-1)}_{l}$ and then
\begin{equation}
\label{form326}
\sum_{i=2}^{+\infty}\frac{1}{i!}ad_{\chi_{l}}^{i}H_{0}=\sum_{i=1}^{+\infty}\frac{1}{(i+1)!}ad_{\chi_{l}}^{i}(\ms{Z}_{l}-\ms{R}^{(l-1)}_{l}).
\end{equation}
By Lemma \ref{le35} and $R_{l}=R_{l-1}(1-\frac{R}{rR_{l-1}})$, one has
\begin{align}
\label{form327}
&\langle|\frac{1}{(i+1)!}ad_{\chi_{l}}^{i}(\ms{Z}_{l}-\ms{R}^{(l-1)}_{l})|\rangle_{s_{0}+l\tau,N}^{R_{l}}
\\\notag\leq&\frac{1}{i+1}\Big(\frac{4(l+2)}{R/(rR_{l-1})}|\chi_{l}|_{s_{0}+l\tau,N}R_{l-1}^{l}\Big)^{i}\langle|\ms{Z}_{l}-\ms{R}_{l}^{(l-1)}|\rangle_{s_{0}+l\tau,N}^{R_{l-1}}.
\end{align}
By \eqref{form322}, one has
\begin{align}
\label{form328}
\frac{4(l+2)}{R/(rR_{l-1})}|\chi_{l}|_{s_{0}+l\tau,N}R_{l-1}^{l}
&\leq8r(l+2)\frac{C_{N}}{\gamma R^{l-1}_{*l-1}}R_{l-1}^{l}
\\\notag&=\frac{8r(l+2)C_{N}R_{*l-1}}{\gamma}(\frac{R_{l-1}}{R_{*l-1}})^{l}
\\\notag&\leq\frac{1}{20r^{2}}(\frac{R_{l}}{R_{*l}})^{l},
\end{align}
where the last inequality follows from $R_{*l-1}\leq2R_{*}\leq\frac{2\gamma}{960r^{4}C_{N}}$ and $\frac{R_{l-1}}{2R_{*l-1}}=\frac{R_{l}}{2R_{*l}}$.
By the construction of $\ms{Z}_{l}$ and the estimate \eqref{form321}, one has
\begin{align}
\label{form329}
\langle|\ms{Z}_{l}-\ms{R}_{l}^{(l-1)}|\rangle_{s_{0}+l\tau,N}^{R_{l-1}}&\leq\langle|\ms{R}_{l}^{(l-1)}|\rangle_{s_{0}+l\tau,N}^{R_{l-1}}
\leq|\ms{R}^{(l-1)}_{l}|_{s_{0}+(l-1)\tau,N}R_{l-1}^{l}
\\\notag&\leq\frac{C_{N}}{R_{*l-1}^{l-1}}R_{l-1}^{l}
\leq C_{N}R_{l-1}.
\end{align}
Hence, by \eqref{form326}--\eqref{form329} one has
\begin{align}
\label{form330}
\langle|\sum_{i=2}^{+\infty}\frac{1}{i!}ad_{\chi_{l}}^{i}H_{0}|\rangle_{s_{0}+l\tau,N}^{R_{l}}
&\leq\sum_{i=1}^{+\infty}\langle|\frac{1}{(i+1)!}ad_{\chi_{l}}^{i}(\ms{Z}_{l}-\ms{R}^{(l-1)}_{l})|\rangle_{s_{0}+l\tau,N}^{R_{l}}
\\\notag&\leq C_{N}R_{l-1}\sum_{i=1}^{+\infty}\frac{1}{i+1}(\frac{1}{20r^{2}})^{i}(\frac{R_{l}}{R_{*l}})^{li}
\\\notag&\leq C_{N}R_{l-1}\frac{1}{20r^{2}}(\frac{R_{l}}{R_{*l}})^{l}
\\\notag&\leq\frac{1}{10r^{2}}C_{N}R_{l}(\frac{R_{l}}{R_{*l}})^{l}.
\end{align}
\\\indent3) Similarly, by Lemma \ref{le35}, \eqref{form315} for $l-1$ and  \eqref{form328}, one has
\begin{align}
\label{form331}
\langle|\sum_{i=1}^{+\infty}\frac{1}{i!}ad_{\chi_{l}}^{i}\ms{Z}^{(l-1)}|\rangle_{s_{0}+l\tau,N}^{R_{l}}
&\leq\langle|\ms{Z}^{(l-1)}|\rangle_{s_{0}+l\tau,N}^{R_{l-1}}
\sum_{i=1}^{+\infty}(\frac{4(l+2)}{R/(rR_{l-1})}|\chi_{l}|_{s_{0}+l\tau,N}R_{l-1}^{l})^{i}
\\\notag&\leq(l-1)C_{N}R_{l-1}\sum_{i=1}^{+\infty}(\frac{1}{20r^{2}})^{i}(\frac{R_{l}}{R_{*l}})^{li}
\\\notag&\leq(l-1)C_{N}R_{l-1}\frac{1}{10r^{2}}(\frac{R_{l}}{R_{*l}})^{l}
\\\notag&\leq \frac{1}{5r}C_{N}R_{l}(\frac{R_{l}}{R_{*l}})^{l}.
\end{align}
\\\indent4) By Lemma \ref{le35}, \eqref{form316} for $l-1$ and \eqref{form328}, one has
\begin{align}
\label{form332}
\langle|\sum_{i=1}^{+\infty}\frac{1}{i!}ad_{\chi_{l}}^{i}\ms{R}^{(l-1)}|\rangle_{s_{0}+l\tau,N}^{R_{l}}
&\leq\langle|\ms{R}^{(l-1)}|\rangle_{s_{0}+l\tau,N}^{R_{l-1}}
\sum_{i=1}^{+\infty}(\frac{4(l+2)}{R/(rR_{l-1})}|\chi_{l}|_{s_{0}+l\tau,N}R_{l-1}^{l})^{i}
\\\notag&\leq C_{N}R_{l-1}(\frac{R_{l-1}}{R_{*l-1}})^{l-1}\sum_{i=1}^{+\infty}(\frac{1}{20r^{2}})^{i}(\frac{R_{l}}{R_{*l}})^{li}
\\\notag&\leq C_{N}R_{l-1}\frac{1}{10r^{2}}(\frac{R_{l}}{R_{*l}})^{l}
\\\notag&\leq\frac{1}{5r^{2}}C_{N}R_{l}(\frac{R_{l}}{R_{*l}})^{l}.
\end{align}
Notice that for any positive $l\leq r$, $1-\frac{1}{2r}+\frac{1}{10r^{2}}+\frac{1}{5r}+\frac{1}{5r^{2}}\leq1$ and by \eqref{form325}, \eqref{form330}--\eqref{form332}, the estimate \eqref{form316} holds.
\end{proof}
%%%%%
\begin{proof}[Proof of Theorem {\sl\ref{th21}}]
After $r$ iterations, let $\phi=\Phi_{\chi_{1}}^{1}\circ\cdots\circ\Phi_{\chi_{r}}^{1}, \ms{Z}=\ms{Z}^{(r)},\ms{R}=\ms{R}^{(r)}$.
By \eqref{form314}, for any $s\geq s_{0}+r\tau+\frac{d+1}{2}$ and $R\leq R_{*}$,  $\phi:B_{s}(\frac{R}{2\rho})\longmapsto B_{s}(\frac{R}{\rho})$ satisfies the estimate:
\begin{align*}
\sup_{\|z\|_{s}\leq R/(2\rho)}\|z-\phi(z)\|_{s}
&\leq\sum_{l=1}^{r} \sup_{\|z\|_{s}\leq R_{l}/(2\rho)}\|z-\Phi^{1}_{\chi_{l}}(z)\|_{s}
\\&\leq\sum_{l=1}^{r}\frac{C_{N}}{2^{l+1}\gamma}R_{l-1}^{2}
\\&\leq\frac{2C_{N}}{\gamma}R^{2}.
\end{align*}
Hence, we can draw the conclusion.
%\begin{equation}
%R_{*}\leq\min\{R',(\frac{\gamma}{4rC_{0}})^{2}\}.
%\end{equation}
%\begin{equation}
%R_{*}\leq(\frac{R^{'r}}{C_{0}})^{2},
%\end{equation}
%and
%If assumptions \eqref{form333} and \eqref{form334} are satisfied, then for any $R\leq R_{*}$, we still have similar results to Lemma \ref{le37}.
\end{proof}
%%%%%
At the end of this section, we will prove long time stability by Theorem \ref{th21}.
\begin{proof}[Proof of Corollary {\sl\ref{co21}}]
Applying Theorem \ref{th21}, for any $\varepsilon\leq\varepsilon_{*}$, there exists a canonical transformation $z=\phi(z')$ with the following estimates:
\begin{equation}
\label{form333}
\sup_{\|z'\|_{s}\leq 5\varepsilon/3}\|z'-\phi(z')\|_{s},
\sup_{\|z\|_{s}\leq 5\varepsilon/3}\|z-\phi^{-1}(z)\|_{s}\leq \frac{2C_{N}}{\gamma}(\frac{10\rho\varepsilon}{3})^{2}.
\end{equation}
So one has
\begin{equation}
\label{form334}
\|z'(0)\|_{s}\leq\|z(0)\|_{s}+\|z(0)-\phi^{-1}(z(0))\|_{s}\leq\varepsilon+\frac{2C_{N}}{\gamma}(\frac{10\rho\varepsilon}{3})^{2}\leq\frac{4}{3}\varepsilon.
\end{equation}
Define $F(z)=\|z\|_{s}^{2}$ and denote by $T$ the escape time of $z'$ from $B_{s}(5\varepsilon/3)$. Then for any $|t|\leq T$,
\begin{align}
\label{form335}|F(z'(t))-F(z'(0))|&=\left|\int_{0}^{t}\{H\circ\phi,F\}(z'(t))dt\right|
\\\notag&=\left|\int_{0}^{t}\{\ms{R},F\}(z'(t))dt\right|
\\\notag&\leq T\sup_{|t|\leq T}\left|\{\ms{R},F\}(z'(t))\right|
\\\notag&\leq T\sup_{|t|\leq T}\|X_{\ms{R}}(z'(t))\|_{s}\|z'(t)\|_{s}.
\end{align}
By Lemma \ref{le32}, \eqref{form217} in Theorem \ref{th21} and the definition of $\varepsilon_{*}$ in \eqref{form219}, one has
\begin{align}
\label{form336}
\sup_{\|z'\|_{s}\leq5\varepsilon/3}\|X_{\ms{R}}(z')\|_{s}
\leq \frac{5\rho}{3}\varepsilon\langle|\ms{R}|\rangle_{s_{0}+r\tau}^{5\rho\varepsilon/3}\leq\frac{C_{N}}{R_{*}^{r}}(\frac{5\rho}{3}\varepsilon)^{r+2}=\frac{25\rho^{2}C_{N}}{9(4\varepsilon_{*})^{r}}\varepsilon^{r+2}.
\end{align}
Notice that for any $|t|\leq T$, $ \|z'(t)\|_{s}\leq\frac{5}{3}\varepsilon$, and by \eqref{form335} and \eqref{form336}, one has
\begin{align}
\label{form337}
|F(z'(t))-F(z'(0))|&\leq T\sup_{\|z'(t)\|_{s}\leq 5\varepsilon/3}\|X_{\ms{R}}(z'(t))\|_{s}\|z'(t)\|_{s}
\\\notag&\leq T\frac{125\rho^{2}C_{N}}{27(4\varepsilon_{*})^{r}}\varepsilon^{r+3}
\\\notag&<\frac{T}{c_{1}}\varepsilon^{r+3}
\end{align}
with the constant $c_{1}$ defined in \eqref{form220}. If $T<c_{1}\varepsilon^{-(r+1)}$, by \eqref{form334} and \eqref{form337}, one has
\begin{align*}
(\frac{5}{3}\varepsilon)^{2}=\|z'(T)\|_{s}^{2}=F(z'(T))\leq F(z'(0))+|F(z'(T))-F(z'(0))|
<(\frac{4}{3}\varepsilon)^{2}+\varepsilon^{2},
\end{align*}
which is impossible. Hence, $T\geq c_{1}\varepsilon^{-(r+1)}$. Then for any $t\leq c_{1}\varepsilon^{-(r+1)}$, by \eqref{form333}, one has
$$\|z(t)\|_{s}\leq\|z'(t)\|_{s}+\|z'(t)-\phi(z'(t))\|_{s}
\leq\frac{5}{3}\varepsilon+\frac{2C_{N}}{\gamma}(\frac{10\rho}{3}\varepsilon)^{2}\leq2\varepsilon.$$
\indent In addition, for any $a\in\mb{Z}^{d}$, one has
\begin{equation}
\label{form338}
|J_{a}(t)-J_{a}(0)|\leq|J_{a}(t)-J'_{a}(t)|+|J'_{a}(t)-J'_{a}(0)|+|J'_{a}(0)-J_{a}(0)|.
\end{equation}
Similarly to \eqref{form335}--\eqref{form337} but $\langle a\rangle^{2s}J'_{a}$ instead of $F(z')$, for any $|t|\leq c_{1}\varepsilon^{-(r+\nu)}$, one has
\begin{equation}
\label{form339}
\langle a\rangle^{2s}|J'_{a}(t)-J'_{a}(0)|\leq\varepsilon^{3-\nu}.
\end{equation}
By \eqref{form333}, one has
\begin{align}
\label{form340}
\langle a\rangle^{2s}|J_{a}(t)-J'_{a}(t)|\leq\big(\|z(t)\|_{s}+\|z'(t)\|_{s}\big)\|z(t)-z'(t)\|_{s}
\leq\frac{100C_{N}\rho^{2}}{\gamma}\varepsilon^{3}.
\end{align}
So \eqref{form223} follows from \eqref{form338}--\eqref{form340}.
\end{proof}

%=========================================================
\section{Applications}
\label{sec4}
In this section, as two examples, we prove long time stability of nonlinear wave equation in one dimension and nonlinear Schr\"{o}dinger equation in high dimension by applying \Cref{th21} and \Cref{co21}.

%=================================
\subsection{Nonlinear wave equation in one dimension}
\label{sec41}
As an application, we consider the following nonlinear wave equation
\begin{equation}
\label{form41}
u_{tt}-u_{xx}+mu+f(x,u)=0,\quad x\in\mb{T},
\end{equation}
where $m\in[1,2]$, the function $f$ is smooth with respect to $x$ and analytic of order at least two with respect to $u$ at the origin, i.e. $f(x,0)=\pa_{u}f(x,u)|_{u=0}=0$.
Write the operator $\Lambda:=(-\partial_{xx}+m)^{1/2}$ and let
\begin{equation}
\label{form42}
z=\frac{1}{\sqrt{2}}(\Lambda^{\frac{1}{2}}u+{\rm i}\Lambda^{-\frac{1}{2}}u_t).
\end{equation}
Then the equation \eqref{form41} is equivalent to
\begin{equation}
\label{form43}
\dot{z}=-{\rm i}\Lambda z-\frac{{\rm i}}{\sqrt{2}}\Lambda^{-\frac{1}{2}}
f\big(x,\Lambda^{-\frac{1}{2}}(\frac{z+\bar{z}}{\sqrt{2}})\big).\\
\end{equation}
Using the Fourier expansion
$z(t,x)=\sum_{a\in\mb{Z}}z_{a}(t)e^{{\rm i}ax}$, rewrite \eqref{form43} as
\begin{equation}
\label{form44}
\dot{z}_{a}=-{\rm i}\frac{\pa H}{\pa \bar{z}_{a}}
\end{equation}
with the Hamiltonian function
\begin{equation}
\label{form45}
H(z,\bar{z})=H_{0}(z,\bar{z})+P(z,\bar{z})=\sum_{a\in\mb{Z}}\omega_{a}|z_{a}|^{2}+\frac{1}{2\pi}\int_{\mb{T}}F\big(x,\sum_{a\in\mb{Z}}(\frac{z_{a}e^{{\rm i}ax}+\bar{z}_{a}e^{-{\rm i}ax}}{\sqrt{2\omega_{a}}})\big)dx,
\end{equation}
where
\begin{equation}
\label{form46}
\omega_{a}=\sqrt{a^{2}+m}
\end{equation}
and $F$ is the primitive function of $f$ with respect to the variable $u$, i.e., $f=\pa_{u}F$.
In view of \eqref{form42}, the norm $\|z\|_{s-\frac{1}{2}}$ is equivalent to the Sobolev norm $\|(u,u_t)\|_{H^{s}(\mb{T})\times H^{s-1}(\mb{T})}$.

%%%%%
For the small initial datum $(u(0,x),u_t(0,x))\in  H^{s}(\mb{T})\times H^{s-1}(\mb{T})$ with large $s$, we have the following result of long time stability.
%%%%%
\begin{theorem}
\label{th41}
Fix a positive integer $r$ and a real number $\tau\geq63(r+2)^{3}$. For any small enough $\gamma>0$, there exists a subset $\Theta\subset[1,2]$ with $|\Theta|\leq e^{12(r+2)}\gamma^{\frac{1}{7(r+2)^{2}}}<1$ such that for any $m\in[1,2]\backslash\Theta$ and $s\geq r\tau+1$, if the initial datum $(u(0,x),\dot{u}(0,x))\in H^{s}(\mb{T})\times H^{s-1}(\mb{T})$ satisfies
$$\varepsilon:=\|(u(0,x),\dot{u}(0,x))\|_{H^{s}(\mb{T})\times H^{s-1}(\mb{T})}\leq\varepsilon_{*},$$
then for $|t|\leq c_{1}\varepsilon^{-r}$, the solution $u(t,x)$ of the nonlinear wave equation \eqref{form41} satisfies
\begin{equation}
\label{form47}
\|u(t,x)\|_{H^{s}(\mb{T})}\leq 2\varepsilon,
\end{equation}
\begin{equation}
\label{form48}\sup_{a\in\mb{Z}}\langle a\rangle^{2s}|J_{a}(t)-J_{a}(0)|\leq c_2\varepsilon^3,
\end{equation}
where $\varepsilon_{*}=\frac{\gamma}{75\times2^{13}r^{4}C_{s+1}}$, $c_{1}=\frac{4^{r-5}\varepsilon_{*}^{r}}{45C_{s+1}}$ and $c_{2}=1+\frac{225\times2^{13}C_{s+1}}{\gamma}$ with the constant $C_{s+1}>0$.
%depend on $r,s,\gamma$.
\end{theorem}
%%%%%
By taking big enough $r$ depending on $\varepsilon$, we get longer time of stability in the following corollary.
%Remark that the result is not optimal.
%%%%%
\begin{corollary}
\label{co41}
Consider the nonlinear wave equation \eqref{form41} with nonlinearity $f(x,u)$ being analytic in $\mb{T}\times\ms{U}$, where  $\ms{U}$ is a neighborhood of the origin.  Assume $0<\lambda<1/4$.  For any small enough $\varepsilon>0$, if the initial datum $(u(0,x),\dot{u}(0,x))$ satisfies $$\|(u(0,x),\dot{u}(0,x))\|_{H^{s}(\mb{T})\times H^{s-1}(\mb{T})}\leq\varepsilon$$
with $s=63(\ln\frac{1}{\varepsilon})^{4\lambda}$, then there exists a subset $\Theta\subset[1,2]$ with $|\Theta|\leq e^{-(\ln\frac{1}{\varepsilon})^{\lambda}}$ such that for any $m\in[1,2]\backslash\Theta$, the solution $u(t,x)$ satisfies
\begin{equation}
\label{form49}
\|u(t,x)\|_{H^{s}(\mb{T})}\leq 2\varepsilon,
\end{equation}
\begin{equation}
\label{form410}\sup_{a\in\mb{Z}}\langle a\rangle^{2s}|J_{a}(t)-J_{a}(0)|\leq \varepsilon^{\frac{5}{2}}
\end{equation}
for 
\begin{equation}
|t|\leq e^{\frac{1}{2}(\ln\frac{1}{\varepsilon})^{1+\lambda}}.
\end{equation}
\end{corollary}
%%%%%
In order to prove Theorem \ref{th41}, we firstly estimate the nonlinearity $P$ in \eqref{form45}.
%%%%%
\begin{lemma}
\label{le41}
There exists $R'>0$ such that for any $N\geq0$, there exists $C_{N}>0$ such that for any $R\leq R'$,
\begin{equation}
\label{form411}
\langle|P|\rangle_{0,N}^{R}\leq C_{N}R.
\end{equation}
Especially, if the nonlinearity $f(x,u)$ is also analytic with respect to $x\in\mb{T}$, then we can take
\begin{equation}
\label{form412}
C_N=C'(\frac{N+2}{e\mu})^{N+2}
\end{equation}
with the constants $\mu,C'>0$.
\end{lemma}
%%%
\begin{proof}
Expand the function $F$ in power series, namely $F(x,u)=\sum_{n=3}^{+\infty}F_{n}(x)u^{n}$.
Write $P=\sum_{n=3}^{+\infty}P_{n}$, where $P_{n}$ is a homogeneous of order $n$,
$$P_{n}=\sum_{b\in\mb{Z}}\hat{F}_{n}(-b)\sum_{\mc{M}(\bs{j})=b}\frac{z_{\bs{j}}}{\sqrt{2}^{n}\sqrt{\omega_{a_{1}}\cdots\omega_{a_{n}}}}$$
with $\hat{F}_{n}(b)$ being the Fourier coefficient of $F_{n}(x)$.
Notice that $F(x,u)$ is smooth with respect to $x$ and analytic with respect to $u$. Hence, there exists $R'>0$ such that for any $N\geq0$, there exists $\tilde{C}_{N}>0$ such that
\begin{equation}
\label{form413}
|\hat{F}_{n}(b)|\leq\frac{\tilde{C}_{N}}{R^{'n}}\langle b\rangle^{-N-2}.
\end{equation}
In view of \eqref{form27}, one has
\begin{align*}
|P_{n}|_{0,N}=\sum_{b\in\mb{Z}}\langle b\rangle^{N}|\hat{F}_{n}(-b)|\sup_{\mc{M}(\bs{j})=b}\big|\frac{1}{2^{\frac{n}{2}}\sqrt{\omega_{a_{1}}\cdots\omega_{a_{n}}}}\big|
\leq\frac{3\tilde{C}_{N}}{2^{\frac{n}{2}}R^{'n}}.
\end{align*}
Hence, in view of \eqref{form28}, for any $R\leq R'$, one has
$$\langle|P|\rangle_{0,N}^{R}=\sum_{n\geq3}|P_{n}|_{0,N}R^{n-2}\leq\frac{3}{2(\sqrt{2}-1)}\frac{\tilde{C}_{N}}{R^{'3}}R.$$
Let $C_{N}=\frac{3}{2(\sqrt{2}-1)}\frac{\tilde{C}_{N}}{R^{'3}}$ and then we  get \eqref{form411}.
When $f(x,u)$ is analytic, there exist constants $C'',\mu>0$ such that for any $n\geq3$,  the Fourier coefficient $\hat{F}_{n}(b)$ satisfies the estimate
$$|\hat{F}_{n}(b)|\leq\frac{C''e^{-\mu|b|}}{R^{'n}}\leq \frac{C''e^{\mu}}{R^{'n}}(\frac{N+2}{e\mu})^{N+2}\langle b\rangle^{-N-2}.$$
Hence, we draw the conclusion with $C'=\frac{3C''e^{\mu}}{2(\sqrt{2}-1)R^{'3}}$.
\end{proof}
%%%%%
Next, we will consider the property of the family of frequencies $\{\omega_{a}\}_{a\in\mb{Z}}$ in \eqref{form46}. For $\tau\geq0$, $\gamma>0$ and $\bs{j}=(\delta_{k},a_{k})_{k=1}^{l}\in(\mb{U}_{2}\times\mb{Z})^{l}\backslash\mc{J}_{l}$ with $l\geq3$, define the resonant set
\begin{equation}
\label{form414}
\ms{P}_{\bs{j}}(\gamma,\tau)=\big\{m\in[1,2] \mid |\delta_{1}\omega_{a_{1}}+\dots+\delta_{l}\omega_{a_{l}}|<\gamma(\frac{j_{1}^{*}}{\langle\mc{M}(\bs{j})\rangle j_{2}^{*}\dots j_{l}^{*}})^{\tau}\big\}.
\end{equation}
Remark that for any permutation $\sigma$,  $\ms{P}_{\sigma(\bs{j})}(\gamma,\tau)=\ms{P}_{\bs{j}}(\gamma,\tau)$. Hence, assume that $j_{k}^{*}=\langle j_{k}\rangle$ without loss of generality in the following.
%have the measure estimate of the resonant set $\ms{P}_{\bs{j}}(\gamma,\tau)$, seeing  Lemma \ref{le7} in Appendix.
%we apply Lemma \ref{le7} and Lemma \ref{le8} in Appendix 
%
Now, we estimate the measure of
\begin{equation}
\label{form415}
\Theta:=\bigcup_{l=3}^{r+2}\bigcup_{\bs{j}\in(\mb{U}_{2}\times\mb{Z})^{l}\backslash\mc{J}_{l}}\ms{P}_{\bs{j}}(\gamma,\tau).
\end{equation}
%%%%%
\begin{lemma}
\label{le42}
If $\tau\geq63(r+2)^{3}$ and $0<\gamma<1$, then
\begin{equation}
\label{form416}
|\Theta|\leq e^{12(r+2)}\gamma^{\frac{1}{7(r+2)^{2}}}.
\end{equation}
\end{lemma}
%%%%%
\begin{proof}
The process is parallel to that in the subsection 6.2 of \cite{B03}. However, for the purpose to obtain concrete constants on the right side of  \eqref{form416}, we need give the details.
Denote $\Theta_{l}:=\bigcup_{\bs{j}\in(\mb{U}_{2}\times\mb{Z})^{l}\backslash\mc{J}_{l}}\ms{P}_{\bs{j}}(\gamma,\tau)$.
We divide the set $\Theta_{l}$ into three parts, namely $\Theta_{l}:=\Theta_{l}^{(1)}\cup\Theta_{l}^{(2)}\cup\Theta_{l}^{(3)}$,
where
\begin{align*}
&\Theta_{l}^{(1)}=\quad\bigcup_{\substack{\bs{j}\in(\mb{U}_{2}\times\mb{Z})^{l}\backslash\mc{J}_{l}\\\delta_{1}\delta_{2}=1}}\ms{P}_{\bs{j}}(\gamma,\tau),
\\&\Theta_{l}^{(2)}=\bigcup_{\substack{\bs{j}\in(\mb{U}_{2}\times\mb{Z})^{l}\backslash\mc{J}_{l}\\\delta_{1}\delta_{2}=-1,j_{2}^{*}\leq A}}\ms{P}_{\bs{j}}(\gamma,\tau),
\\&\Theta_{l}^{(3)}=\bigcup_{\substack{\bs{j}\in(\mb{U}_{2}\times\mb{Z})^{l}\backslash\mc{J}_{l}\\\delta_{1}\delta_{2}=-1,j_{2}^{*}>A}}\ms{P}_{\bs{j}}(\gamma,\tau)
\end{align*}
with taking
\begin{equation}
\label{form417}
A=\gamma^{-\frac{1}{7l}}(\prod_{k=3}^{l}j_{k}^{*})^{\frac{\tau}{21l}}.
\end{equation}
(i) Estimate of $\Theta_{l}^{(1)}$.
\\\indent In view of $\delta_{1}\delta_{2}=1$ and the fact $\langle a_{k}\rangle/\sqrt2\leq\omega_{a_{k}}\leq\sqrt{2}\langle a_{k}\rangle$, for any $m\in\ms{P}_{\bs{j}}(\gamma,\tau)$ with $\gamma<1$, one has
$$j_{1}^{*}+j_{2}^{*}\leq\sqrt{2}(\omega_{a_1}+\omega_{a_2})\leq\sqrt{2}(1+\omega_{a_{3}}+\cdots+\omega_{a_{l}})\leq2(1+j_{3}^{*}+\cdots+j_{l}^{*})\leq2l\prod_{k=3}^{l}j_{k}^{*}.$$
Hence, $j_{1}^{*}\leq2l\prod_{k=3}^{l}j_{k}^{*}$ and $j_{2}^{*}\leq l\prod_{k=3}^{l}j_{k}^{*}.$
By Lemma \ref{le7} in Appendix and $\tau\geq9l^{2}+30l$, one has
 \begin{equation}
 \label{form418}
|\ms{P}_{\bs{j}}(\gamma,\tau)|\leq2^{3l+6}l^{9}(j_{2}^{*})^{4}(\prod_{k=3}^{l}j_{k}^{*})^{3l+2-\frac{\tau}{3l}}\gamma^{\frac{1}{l}}
\leq2^{3l+6}l^{13}(\prod_{k=3}^{l}j_{k}^{*})^{-4}\gamma^{\frac{1}{l}}.
\end{equation}
Notice that
\begin{equation}
\label{form419}
\sum_{j\in\mb{U}_{2}\times\mb{Z}}\langle j\rangle^{-2}\leq\frac{2\pi^{2}}{3}.
\end{equation}
Taking the sum about $\bs{j}\in(\mb{U}_{2}\times\mb{Z})^{l}\backslash\mc{J}_{l}$ with $\delta_{1}\delta_{2}=1$, by \eqref{form418} and \eqref{form419}, one has
 \begin{align}
\label{form420}
|\Theta_{l}^{(1)}|\leq2^{3l+11}l^{15}(\frac{2\pi^{2}}{3})^{l-2}\gamma^{\frac{1}{l}}.
\end{align}
(ii) Estimate of $\Theta_{l}^{(2)}$.
\\\indent In view of $\delta_{1}\delta_{2}=-1$ and $j_{2}^{*}\leq A$, for any $m\in\ms{P}_{\bs{j}}(\gamma,\tau)$ with $\gamma<1$, one has
$$j_{1}^{*}\leq\sqrt{2}\,\omega_{a_1}\leq\sqrt{2}(1+\omega_{a_2}+\cdots+\omega_{a_l})\leq2(1+j_{2}^{*}+\cdots j_{l}^{*})\leq2lA\prod_{k=3}^{l}j_{k}^{*}.$$
By Lemma \ref{le7} in Appendix and \eqref{form417}, one has
\begin{align*}
|\Theta_{l}^{(2)}|&\leq\sum_{\substack{\bs{j}\in(\mb{U}_{2}\times\mb{Z})^{l}\\j_{2}^{*}\leq A,j_{1}^{*}\leq2lAj_{3}^{*}\cdots j_{l}^{*}}}2^{3l+6}l^{9}(j_{2}^{*})^{4}(\prod_{k=3}^{l}j_{k}^{*})^{3l+2-\frac{\tau}{3l}}\gamma^{\frac{1}{l}}
\\&\leq2^{3l+11}l^{10}\sum_{j_{3},\dots,j_{l}}A^{6}(\prod_{k=3}^{l}j_{k}^{*})^{3l+3-\frac{\tau}{3l}}\gamma^{\frac{1}{l}}
\\&=2^{3l+11}l^{10}\sum_{j_{3},\dots,j_{l}}(\prod_{k=3}^{l}j_{k}^{*})^{3l+3-\frac{\tau}{21l}}\gamma^{\frac{1}{7l}}.
\end{align*}
Then by $\tau\geq63l^{2}+105l$ and \eqref{form419}, one has
\begin{equation}
\label{form421}
|\Theta_{l}^{(2)}|\leq2^{3l+11}l^{10}\sum_{j_{3},\dots,j_{l}}(\prod_{k=3}^{l}j_{k}^{*})^{-2}\gamma^{\frac{1}{7l}}\leq2^{3l+11}l^{10}(\frac{2\pi^{2}}{3})^{l-2}\gamma^{\frac{1}{7l}}.
\end{equation}
(iii) Estimate of $\Theta_{l}^{(3)}$.
\\\indent In view of $\delta_{1}\delta_{2}=-1$ and $j_{2}^{*}>A$, one has $\delta_{1}\omega_{a_{1}}+\delta_{2}\omega_{a_{2}}=n+r_{a_{1},a_{2}}$ with $n=\delta_{1}(|a_{1}|-|a_{2}|)$ and
\begin{equation}
\label{form422}
|r_{a_{1},a_{2}}|\leq\frac{2}{\langle a_{2}\rangle}\leq\frac{2}{A}
\end{equation}
for any $m\in[1,2]$.
By Lemma \ref{le3} in Appendix, one has
\begin{equation}
\label{form423}
\gamma(\frac{j_{1}^{*}}{\langle\mc{M}(\bs{j})\rangle j_{2}^{*}\dots j_{l}^{*}})^{\tau}\leq\gamma(j_{3}^{*}\dots j_{l}^{*})^{-\frac{\tau}{3}}\leq\frac{1}{A}.
\end{equation}
Write $\bs{j}':=(\delta_{k},a_{k})_{k=3}^{l}$ and let
\begin{equation}
\label{form424}
\ms{P}'_{n,\bs{j}'}(\frac{3}{A}):=\big\{m\in[1,2] \mid |n+\delta_{3}\omega_{a_{3}}+\dots+\delta_{l}\omega_{a_{l}}|<\frac{3}{A}\big\}.
\end{equation}
By \eqref{form422} and \eqref{form423}, one has $\ms{P}_{\bs{j}}(\gamma,\tau)\subseteq\ms{P}'_{n,\bs{j}'}(\frac{3}{A})$ and thus $\Theta_{l}^{(3)}\subseteq\bigcup_{n,\bs{j}'}\ms{P}'_{n,\bs{j}'}(\frac{3}{A})$.
Notice that
$$|n|\leq|r_{a_1,a_2}|+1+\omega_{a_3}+\cdots+\omega_{a_l}\leq\frac{2}{A}+1+\sqrt{2}(j_{3}^{*}+\cdots+j_{l}^{*})\leq2l\prod_{k=3}^{l}j_{k}^{*}.$$
Then by Lemma \ref{le8} in Appendix, one has
\begin{align*}
|\Theta_{l}^{(3)}|&\leq\sum_{n,\;j_{3},\dots,j_{l}}2^{3l-3}l^{6}(\prod_{k=3}^{l}j_{k}^{*})^{3l-6}(\frac{3}{A})^{\frac{1}{l}}
\\&\leq2^{3l-1}l^{7}\sum_{j_{3},\dots,j_{l}}(\prod_{k=3}^{l}j_{k}^{*})^{3l-5}(\frac{3}{A})^{\frac{1}{l}}
\\&\leq2^{3l}l^{7}\sum_{j_{3},\dots,j_{l}}(\prod_{k=3}^{l}j_{k}^{*})^{3l-5-\frac{\tau}{21l^{2}}}\gamma^{\frac{1}{7l^{2}}}.
\end{align*}
 Hence, by  $\tau\geq63l^{3}-63l^{2}$ and \eqref{form419}, one has
\begin{equation}
\label{form425}
|\Theta_{l}^{(3)}|\leq2^{3l}l^{7}\sum_{j_{3},\dots,j_{l}}(\prod_{k=3}^{l}j_{k}^{*})^{-2}\gamma^{\frac{1}{7l^{2}}}\leq2^{3l}l^{7}(\frac{2\pi^{2}}{3})^{l-2}\gamma^{\frac{1}{7l^{2}}}.
\end{equation}
\indent To sum up, when $\tau\geq63l^{3}$, by estimates \eqref{form420}, \eqref{form421} and \eqref{form425}, one has
\begin{align}
\label{form426}
|\Theta_{l}|&\leq|\Theta_{l}^{(1)}|+|\Theta_{l}^{(2)}|+|\Theta_{l}^{(3)}|
\\\notag&\leq2^{3l+11}l^{15}(\frac{2\pi^{2}}{3})^{l-2}\gamma^{\frac{1}{l}}+2^{3l+11}l^{10}(\frac{2\pi^{2}}{3})^{l-2}\gamma^{\frac{1}{7l}}+2^{3l}l^{7}(\frac{2\pi^{2}}{3})^{l-2}\gamma^{\frac{1}{7l^{2}}}
%\\\notag&\leq(2^{5}l^{9}+2^{10}l^{9}+2^{4}l^{6})(\frac{16\pi^{2}}{3})^{l-2}\gamma^{\frac{1}{6l^{2}}}
\\\notag&\leq2^{18}l^{15}(\frac{16\pi^{2}}{3})^{l-2}\gamma^{\frac{1}{7l^{2}}}
\\\notag&\leq e^{9l+6}\gamma^{\frac{1}{7l^{2}}},
\end{align}
 where the last inequality follows from $2^{18}l^{15}\leq e^{5l+14}$ and $16\pi^{2}/3\leq e^{4}$.
Hence, when $\tau\geq63(r+2)^{3}$, one has
\begin{equation}
\label{form427}
|\Theta|=|\bigcup_{l=3}^{r+2}\Theta_{l}|\leq\sum_{l=3}^{r+2}e^{9l+6}\gamma^{\frac{1}{7l^{2}}}\leq e^{9(r+2)+7}\gamma^{\frac{1}{7(r+2)^{2}}}.
\end{equation}
The proof is completed.
\end{proof}
%%%%%
\begin{proof}[Proof of Theorem {\sl\ref{th41}}]
By Lemma \ref{le41}, there exists $R'>0$ such that for any $N\geq0$, there exists $C_{N}>0$ such that $R\leq R'$, one has $\langle|P|\rangle_{0,N}^{R}\leq C_{N}R$. This means that the condition \eqref{form211} in Theorem \ref{th21} holds with $s_0=0$.
Notice that $\rho=96$ in \eqref{form212}. Assuming $\gamma<480r^{4}C_{N}R'$, then $R_{*}=\frac{\gamma}{960r^{4}C_{N}}$ in \eqref{form213}.
Then for any $s\geq r\tau+1$ and $N=s+1$, one has $\varepsilon_*=\frac{\gamma}{75\times2^{13}r^{4}C_{N}}$ in \eqref{form219}.
Assuming $\gamma<e^{-84(r+2)^{3}}$, then by Lemma \ref{le42} and Corollary \ref{co21}, we can get Theorem \ref{th41} with $c_{1}=\frac{4^{r-5}\varepsilon_{*}^{r}}{45C_{N}}$ and $c_{2}=1+\frac{225\times2^{13}C_{N}}{\gamma}$.
\end{proof}
%%%%%
Finally, we apply Theorem \ref{th41} to prove corollary \ref{co41}.
%%%%%
\begin{proof}[Proof of Corollary {\sl\ref{co41}}]
Take $r=[(\ln\frac{1}{\varepsilon})^{\lambda}]-2$, $\tau=63(r+2)^{3}$ and $\gamma=e^{-91(\ln\frac{1}{\varepsilon})^{3\lambda}}$.
Then $s=63(\ln\frac{1}{\varepsilon})^{4\lambda}\geq r\tau+1$ and
$$|\Theta|\leq e^{12(r+2)}\gamma^{\frac{1}{7(r+2)^{2}}}\leq e^{-(\ln\frac{1}{\varepsilon})^{\lambda}}.$$
By \eqref{form412} with $N=s+1$ in Lemma \ref{le41}, there exist the constants $\mu,C'>0$ such that
 \begin{equation}
 \label{form428}
 c_{1}\varepsilon^{-r}=\tilde{c}_{1}(\frac{e\mu}{N+2})^{(N+2)(r+1)}(\frac{\tilde{c}_{0}\gamma }{r^{4}})^{r}\varepsilon^{-r},
\end{equation}
\begin{equation}
\label{form429}
c_{2}\varepsilon^{3}=\big(1+\frac{\tilde{c}_2}{\gamma}(\frac{N+2}{e\mu})^{N+2}\big)\varepsilon^{3}
 \end{equation}
 with $\tilde{c}_{0}=\frac{1}{75\times2^{11}C'}$, $\tilde{c}_1=\frac{1}{45\times2^{10}C'}$ and $\tilde{c}_2=225\times2^{13}C'$.
Notice that for $0<\lambda<1/4$ and small enough $\varepsilon$, one has
 \begin{equation}
  \label{form430}
 (N+2)(r+1)\ln\frac{N+2}{e\mu}\leq\frac{1}{6}r\ln\frac{1}{\varepsilon},
 \end{equation}
\begin{equation}
 \label{form431}\ln\frac{r^{4}}{\tilde{c}_0\gamma}\leq\frac{1}{6}\ln\frac{1}{\varepsilon}.
 \end{equation}
By \eqref{form428}, \eqref{form430} and \eqref{form431}, one has
\begin{align*}
c_{1}\varepsilon^{-r}&\geq\tilde{c}_{1}\varepsilon^{-\frac{2}{3}r}\geq\varepsilon^{-\frac{1}{2}(\ln\frac{1}{\varepsilon})^{\lambda}}=e^{\frac{1}{2}(\ln\frac{1}{\varepsilon})^{1+\lambda}}.
\end{align*}
Similarly, by \eqref{form429}, one has
\begin{align*}
c_{2}\varepsilon^{3}\leq\varepsilon^{\frac{5}{2}}.
\end{align*}
 Hence, we can draw the conclusion.
%
%By Theorem \ref{th41}, for any $s\geq63(\ln\frac{1}{\varepsilon})^{4\lambda}$,  there exists a subset $\Theta\subset[1,2]$ with $|\Theta|\leq e^{-\frac{1}{7}(\ln\frac{1}{\varepsilon})^{\lambda}}$ such that for any $|t|\leq \varepsilon^{\frac{1}{2}(\ln\frac{1}{\varepsilon})^{\lambda}}$, \eqref{form43} and \eqref{form43'} hold.
\end{proof}

%=================================
\subsection{Nonlinear Schr\"{o}dinger equation in high dimension}
\label{sec42}
Consider the nonlinear Schr\"{o}dinger equation on the $d$-dimensional torus
\begin{equation}
\label{form432}
{\rm i}\pa_{t}u=-\Delta u+V*u+g(x,u,\bar{u}),\quad x\in\mb{T}^{d},
\end{equation}
where $d\geq2$,  the real value function $g$ is smooth with respect to $x$ and analytic of order at least three with respect to $u,\bar{u}$ at $(u,\bar{u})=(0,0)$.
The potential $V(x)=\sum_{a\in\mb{Z}^{d}}\frac{v_{a}}{\langle a\rangle^{m}}e^{{\rm i}a\cdot x}$, where $m>\frac{d}{2}$ is fixed, and $\{v_{a}\}_{a\in\mb{Z}^{d}} $ is chosen in the space $\ms{V}$ given by
\begin{equation}
\label{form433}
\ms{V}=\big\{\{v_{a}\}_{a\in\mb{Z}^{d}} \mid v_{a}\in[-\frac{1}{2},\frac{1}{2}]\;\text{for}\;a\in\mb{Z}^{d}\big\}.
\end{equation}
%%%
\\\indent Expand $u(x)=\sum_{a\in\mb{Z}^{d}}z_{a}e^{{\rm i}a\cdot x}$ and then the Hamiltonian function is given by
\begin{equation}
\label{form434}
H(z,\bar{z})=H_{0}(z,\bar{z})+P(z,\bar{z})
=\sum_{a\in\mb{Z}^{d}}\omega_{a}|z_{a}|^{2}+\frac{1}{(2\pi)^{d}}\int_{\mb{T}^{d}}G(x,u,\bar{u})dx,
\end{equation}
where
\begin{equation}
\label{form435}
\omega_{a}=|a|^{2}+\frac{v_{a}}{\langle a\rangle^{m}}
\end{equation}
and $G$ is the primitive function of $g$ with respect to the variable $\bar{u}$, i.e., $g=\pa_{\bar{u}}G$.
\\\indent Recall the notation $\rho=2^{d+4}3^{d}$ in \eqref{form212}, and in this subsection, we fix a real number $\tau\geq15(m+2d+1)$.  Then we have long time stability of solutions of the equation \eqref{form432}.
%%%%%
\begin{theorem}
\label{th42}
Fix a positive integer $r$. For any small enough $\gamma>0$, there exists a subset $\ms{V}'\subset\ms{V}$ satisfying $|\ms{V}'|\leq\rho^{r+3}\gamma^{\frac{m}{m+2d}}<1$ such that for any $\{v_a\}_{a\in\mb{Z}^{d}}\in\ms{V}\backslash\ms{V}'$ and $s\geq r\tau+\frac{d+1}{2}$, if the initial datum $u(0,x)\in H^{s}(\mb{T}^{d})$ satisfies
$\varepsilon:=\|u(0,x)\|_{H^{s}(\mb{T}^{d})}\leq\varepsilon_{*}$,
then for $|t|\leq c_{1}\varepsilon^{-r}$, the solution $u(t,x)$ of the nonlinear Schr\"{o}dinger equation \eqref{form432} satisfies
\begin{equation}
\|u(t)\|_{H^{s}(\mb{T}^{d})}\leq 2\varepsilon,
\end{equation}
\begin{equation}
\sup_{a\in\mb{Z}^{d}}\langle a\rangle^{2s}|I_{a}(t)-I_{a}(0)|\leq c_2\varepsilon^{3},
\end{equation}
where $\varepsilon_{*}=\frac{3\gamma}{200\rho^{2}r^4C_{s+1}}$, $c_{1}=\frac{(4\varepsilon_{*})^{r}}{5\rho^{2}C_{s+1}}$ and $c_{2}=1+\frac{200\rho^2C_{s+1}}{\gamma}$ with the constant $C_{s+1}>0$.
\end{theorem}
%%%%%
Take big enough $r$ depending on $\varepsilon$, and then we get longer time of stability.
%%%%%
\begin{corollary}
\label{co42}
Consider the nonlinear Schr\"{o}dinger equation \eqref{form432} with nonlinearity $g(x,u,\bar{u})$ being analytic in $\mb{T}\times\ms{U}$, where $\ms{U}$ is a neighborhood of the origin. For any small enough $\varepsilon>0$, if the initial datum $u(0,x)$ satisfies
$\|u(0,x)\|_{H^{s}(\mb{T}^{d})}\leq\varepsilon$,
where $s=\frac{\log_{\rho}\frac{1}{\varepsilon}}{12\log_{\rho}\log_{\rho}\frac{1}{\varepsilon}}$,
then there exists a subset $\ms{V}'\subset\ms{V}$ satisfying $$|\ms{V}'|\leq \rho^{-\frac{\log_{\rho}\frac{1}{\varepsilon}}{60\tau\log_{\rho}\log_{\rho}\frac{1}{\varepsilon}}}$$
such that for any $\{v_a\}_{a\in\mb{Z}^{d}}\in\ms{V}\backslash\ms{V}'$, the solution $u(t,x)$ satisfies
\begin{equation}
\|u(t)\|_{H^{s}(\mb{T}^{d})}\leq 2\varepsilon,
\end{equation}
\begin{equation}
\sup_{a\in\mb{Z}^{d}}\langle a\rangle^{2s}|I_{a}(t)-I_{a}(0)|\leq \varepsilon^{\frac{5}{2}}
\end{equation}
for 
\begin{equation}
\label{form8-5-3}
|t|\leq\rho^{\frac{(\log_{\rho}\frac{1}{\varepsilon})^{2}}{24\tau\log_{\rho}\log_{\rho}\frac{1}{\varepsilon}}}.
\end{equation}
\end{corollary}
%%%%%
Especially, if nonlinearity of  the equation \eqref{form432} does not contain the spatial variable $x$ explicitly, then the denominator $\log_{\rho}\log_{\rho}\frac{1}{\varepsilon}$ in the stability time \eqref{form8-5-3} could be deleted. See the following corollary.
%%%%%
\begin{corollary}
\label{re41}
Consider the nonlinear Schr\"{o}dinger equation \eqref{form432} with nonlinearity $g(u,\bar{u})$ being analytic in $\ms{U}$, where $\ms{U}$ is a neighborhood of the origin.
For any $\varepsilon>0$ small enough and $s\geq\frac{\tau}{9}\log_{\rho}\frac{1}{\varepsilon}$, if the initial datum $u(0,x)$ satisfies
$\|u(0,x)\|_{H^{s}(\mb{T}^{d})}\leq\varepsilon,$
then there exists a subset $\ms{V}'\subset\ms{V}$ satisfying $|\ms{V}'|\leq\varepsilon^{\frac{1}{30}}$ such that for any $\{v_a\}_{a\in\mb{Z}^{d}}\in\ms{V}\backslash\ms{V}'$, the solution $u(t,x)$ satisfies
\begin{equation}
\|u(t)\|_{H^{s}(\mb{T}^{d})}\leq 2\varepsilon,
\end{equation}
\begin{equation}
\sup_{a\in\mb{Z}^{d}}\langle a\rangle^{2s}|I_{a}(t)-I_{a}(0)|\leq \varepsilon^{\frac{5}{2}}
\end{equation}
for
\begin{equation}
|t|\leq\rho^{\frac{1}{46}(\log_{\rho}\frac{1}{\varepsilon})^{2}}.
\end{equation}
\end{corollary}
%%%%%
\indent Similarly with Lemma \ref{le41}, we prove that nonlinearity $P(z,\bar{z})$ of \eqref{form434} satisfies the assumption \eqref{form211}.
%%%%%
\begin{lemma}
\label{le43}
There exists $R'>0$ such that for any $N\geq0$, there exists $C_{N}>0$ such that for any $R\leq R'$,
\begin{equation}
\label{form436}
\langle|P|\rangle_{0,N}^{R}\leq C_{N}R.
\end{equation}
Especially, if the nonlinearity $g(x,u,\bar{u})$ is also analytic with respect to $x\in\mb{T}$, then we can take
\begin{equation}
\label{form437}
C_N=C'(\frac{N+2}{e\mu})^{N+2}
\end{equation}
with the constants $\mu, C'>0$.
\end{lemma}
%%%
\begin{proof}
Expand the function $G$ in power series, namely
$G(x,u)=\sum_{n=3}^{+\infty}G_{n}(x)u^{n}$.
Write $P=\sum_{n=3}^{+\infty}P_{n}$, where $P_{n}$ is homogeneous of order $n$,
$$P_{n}=\sum_{b\in\mb{Z}}\hat{G}_{n}(-b)\sum_{\mc{M}(\bs{j})=b} z_{\bs{j}}$$
with $\hat{G}_{n}(b)$ being the Fourier coefficient of $G_{n}(x)$.
Notice that $G(x,u,\bar{u})$ is smooth with respect to $x$ and analytic with respect to $u,\bar{u}$. Hence, there exists $R'>0$ such that for any $N\geq0$, there exists $\tilde{C}_{N}>0$ such that
\begin{equation}
\label{form438}
|\hat{G}_{n}(b)|\leq\frac{\tilde{C}_{N}}{(2R')^{n}}\langle b\rangle^{-N-2}.
\end{equation}
In view of \eqref{form27}, one has
\begin{align*}
|P_{n}|_{0,N}=\sum_{b\in\mb{Z}^{d}}\langle b\rangle^{N}|\hat{G}_{n}(-b)|\leq\frac{3\tilde{C}_{N}}{(2R')^{n}}.
\end{align*}
Hence, in view of \eqref{form28}, for any $R\leq R'$, one has
$$\langle|P|\rangle_{0,N}^{R}=\sum_{n\geq3}|P_{n}|_{0,N}R^{n-2}\leq\frac{3\tilde{C}_{N}}{4R^{'3}}R.$$
Let $C_{N}=\frac{\tilde{3C_{N}}}{4R^{'3}}$ and then we get \eqref{form436}.
When $g(x,u,\bar{u})$ is analytic, there exist constants $C'',\mu>0$ such that for any $n\geq3$,  the Fourier coefficient $\hat{G}_{n}(b)$ satisfies the estimate
$$|\hat{G}_{n}(b)|\leq\frac{C''e^{-\mu|b|}}{(2R')^{n}}\leq \frac{C''e^{\mu}}{(2R')^{n}}(\frac{N+2}{e\mu})^{N+2}\langle b\rangle^{-N-2}.$$
Hence, we draw the conclusion with $C'=\frac{3C''e^{\mu}}{4R^{'3}}$.
\end{proof}
%%%%%
%So the Theorem \ref{th21} is applicable for the nonlinear Schr\"{o}dinger equation \eqref{form422}.
The following Lemma shows that for almost all $V\in\ms{V}$, the family of frequencies $\{\omega_{a}\}_{a\in\mb{Z}^{d}}$ in \eqref{form435} is non-resonant up to order $r+2$.
%%%%%
\begin{lemma}
\label{le44}
For any $\bs{j}=(\delta_{k},a_{k})_{k=1}^{l}\in(\mb{U}_{2}\times\mb{Z}^{d})^{l}\backslash\mc{I}_{l}$, define the resonant set
$$\ms{V}_{\bs{j}}(\gamma,\tau)=\big\{\{v_{a}\}_{a\in\mb{Z}^{d}}\in[-\frac{1}{2},\frac{1}{2}]^{\mb{Z}^{d}} \mid |\delta_{1}\omega_{a_{1}}+\dots+\delta_{l}\omega_{a_{l}}|<\gamma(\frac{j_{1}^{*}}{\langle\mc{M}(\bs{j})\rangle j_{2}^{*}\dots j_{l}^{*}})^{\tau}\big\}.$$
If $\tau\geq15(m+2d+1)$ and $0<\gamma<1$, then
\begin{equation}
\label{form439}
|\bigcup_{l=3}^{r+2}\bigcup_{\bs{j}\in(\mb{U}_{2}\times\mb{Z}^{d})^{l}\backslash\mc{I}_{l}}\ms{V}_{\bs{j}}(\gamma,\tau)|\leq\rho^{r+3}\gamma^{\frac{m}{m+2d}}
\end{equation}
with $\rho=2^{d+4}3^{d}$.
\end{lemma}
%%%
\begin{proof}
Remark that for any permutation $\sigma$,  $\ms{V}_{\sigma(\bs{j})}(\gamma,\tau)=\ms{V}_{\bs{j}}(\gamma,\tau)$. Hence, assume that $j_{k}^{*}=\langle j_{k}\rangle$ without loss of generality.
In view of $\bs{j}\in(\mb{U}_{2}\times\mb{Z}^{d})^{l}\backslash\mc{I}_{l}$, by Lemma \ref{le3} in Appendix, one has
\begin{equation}
\label{form440}
|\ms{V}_{\bs{j}}(\gamma,\tau)|\leq2(j_{l}^{*})^{m}\gamma(\frac{j_{1}^{*}}{\langle\mc{M}(\bs{j})\rangle j_{2}^{*}\dots j_{l}^{*}})^{\tau}\leq2\gamma(\prod_{k=3}^{l}j_{k}^{*})^{m-\frac{\tau}{3}}.
\end{equation}
Similarly with the proof of Lemma \ref{le42}, we divide $\bigcup_{\bs{j}\in(\mb{U}_{2}\times\mb{Z}^{d})^{l}\backslash\mc{I}_{l}}\ms{V}_{\bs{j}}(\gamma,\tau)$ into three parts, namely $\bigcup_{\bs{j}\in(\mb{U}_{2}\times\mb{Z}^{d})^{l}\backslash\mc{I}_{l}}\ms{V}_{\bs{j}}(\gamma,\tau)=\ms{V}^{(1)}_{l}+\ms{V}^{(2)}_{l}+\ms{V}^{(3)}_{l}$, where
\begin{align*}
&\ms{V}^{(1)}_{l}=\bigcup_{\substack{\bs{j}\in(\mb{U}_{2}\times\mb{Z}^{d})^{l}\backslash\mc{I}_{l}\\\delta_{1}\delta_{2}=1}}\ms{V}_{\bs{j}}(\gamma,\tau),
\\&\ms{V}^{(2)}_{l}=\bigcup_{\substack{\bs{j}\in(\mb{U}_{2}\times\mb{Z}^{d})^{l}\backslash\mc{I}_{l}\\\delta_{1}\delta_{2}=-1,j_{2}^{*}\leq A}}\ms{V}_{\bs{j}}(\gamma,\tau),
\\&\ms{V}^{(3)}_{l}=\bigcup_{\substack{\bs{j}\in(\mb{U}_{2}\times\mb{Z}^{d})^{l}\backslash\mc{I}_{l}\\\delta_{1}\delta_{2}=-1,j_{2}^{*}>A}}\ms{V}_{\bs{j}}(\gamma,\tau)
\end{align*}
with taking
\begin{equation}
\label{form441}
A=\gamma^{-\frac{1}{m+2d}}(\prod_{k=3}^{l}j_{k}^{*})^{\frac{\tau}{3(m+2d)}}.
\end{equation}
(i) Estimate of $\ms{V}_{l}^{(1)}$.
\\\indent In view of $\delta_{1}\delta_{2}=1$, for any $\{v_{a}\}_{a\in\mb{Z}^{d}}\in\ms{V}_{\bs{j}}(\gamma,\tau)$ with $\gamma<1$, one has $|a_1|^{2}+|a_2|^{2}\leq|a_3|^{2}+\cdots+|a_l|^{2}+1+\frac{l}{2}$ and then
$$(j_{1}^{*})^{2}\leq2(1+|a_1|^{2})\leq2(1+|a_3|^{2}+\cdots+|a_l|^{2}+1+\frac{l}{2})\leq3l(\prod_{k=3}^{l}j_{k}^{*})^{2}.$$
Hence, one has $j_{2}^{*}\leq j_{1}^{*}\leq\sqrt{3l}\prod_{k=3}^{l}j_{k}^{*}$. Taking the sum about $\bs{j}\in(\mb{U}_{2}\times\mb{Z}^{d})^{l}\backslash\mc{I}_{l}$ with $\delta_{1}\delta_{2}=1$, by \eqref{form440}, the fact $\tau\geq3m+9d+3$ and Lemma \ref{le1} in Appendix, one has
\begin{align}
\label{form442}
|\ms{V}_{l}^{(1)}|&\leq4\sum_{j_{3},\cdots,j_{l}}2\gamma(\prod_{k=3}^{l}j_{k}^{*})^{m-\frac{\tau}{3}}\big(2\sqrt{3l}\prod_{k=3}^{l}j_{k}^{*}\big)^{2d}
\\\notag&\leq8(12l)^{d}\sum_{j_{3},\cdots,j_{l}}\gamma(\prod_{k=3}^{l}j_{k}^{*})^{-(d+1)}
\\\notag&\leq8(12l)^{d}(2\times3^{d})^{l-2}\gamma.
\end{align}
(ii)  Estimate of $\ms{V}_{l}^{(2)}$.
\\\indent In view of $\delta_{1}\delta_{2}=-1$ and $j_{2}^{*}\leq A$, for any $\{v_a\}_{a\in\mb{Z}^{d}}\in\ms{V}_{\bs{j}}(\gamma,\tau)$ with $\gamma<1$, one has
$|a_1|^{2}\leq|a_2|^{2}+\cdots+|a_l|^{2}+1+\frac{l}{2}$ and then
$$(j_{1}^{*})^{2}\leq2(1+|a_1|^{2})\leq2(1+|a_2|^{2}+\cdots+|a_l|^{2}+1+\frac{l}{2})\leq (lA\prod_{k=3}^{l}j_{k}^{*})^{2}.$$
Hence, $j_{1}^{*}\leq lA\prod_{k=3}^{l}j_{k}^{*}$.
By \eqref{form440}, \eqref{form441},the fact $\tau\geq3(m+2d+1)(m+2d)/m$ and Lemma \ref{le1} in Appendix, one has
\begin{align}
\label{form443}
|\ms{V}_{l}^{(2)}|&\leq 4\sum_{j_{3},\cdots,j_{l}}2\gamma(\prod_{k=3}^{l}j_{k}^{*})^{m-\frac{\tau}{3}}(2A)^{d}(2lA\prod_{k=3}^{l}j_{k}^{*})^{d}
\\\notag&\leq8(4l)^{d}\sum_{j_{3},\cdots,j_{l}}(\prod_{k=3}^{l}j_{k}^{*})^{m+d-\frac{\tau m}{3(m+2d)}}\gamma^{\frac{m}{m+2d}}
\\\notag&\leq8(4l)^{d}(2\times3^{d})^{l-2}\gamma^{\frac{m}{m+2d}}.
\end{align}
(iii) Estimate of $\ms{V}_{l}^{(3)}$.
\\\indent In view of $\delta_{1}\delta_{2}=-1$ and $j_{2}^{*}>A$, one has $\omega_{a_{1}}-\omega_{a_{2}}=n+\nu_{a_{1},a_{2}}$ with $n=\delta_{1}(|a_{1}|^{2}-|a_{2}|^{2})$ and
\begin{equation}
\label{form444}
|\nu_{a_{1},a_{2}}|=\big|\frac{v_{a_{1}}}{\langle a_{1}\rangle^{m}}-\frac{v_{a_{2}}}{\langle a_{2}\rangle^{m}}\big|\leq\frac{1}{\langle a_{2}\rangle^{m}}<\frac{1}{A^{m}}.
\end{equation}
By Lemma \ref{le3} in Appendix and \eqref{form441}, one has
\begin{equation}
\label{form445}
\gamma(\frac{j_{1}^{*}}{\mc{M}(\bs{j})j_{2}^{*}\dots j_{l}^{*}})^{\tau}\leq\gamma(j_{3}^{*}\dots j_{l}^{*})^{-\frac{\tau}{3}}\leq\frac{1}{A^{m}}.
\end{equation}
Write $\bs{j}'=(j_{3},\cdots,j_{l})$ and let
$$\ms{V}'_{n,\bs{j}'}(A')=\big\{\{v_{a}\}_{a\in\mb{Z}^{d}}\in[-\frac{1}{2},\frac{1}{2}]^{\mb{Z}^{d}} \mid |n+\delta_{3}\omega_{a_{3}}+\dots+\delta_{l}\omega_{a_{l}}|<\frac{2}{A^{m}}\big\}.$$
By \eqref{form444} and \eqref{form445}, one has $\ms{V}_{\bs{j}}(\gamma,\tau)\subseteq\ms{V}'_{n,\bs{j}'}(A)$ and thus $\ms{V}_{l}^{(3)}\subseteq\bigcup_{n,\bs{j}'}\ms{V}'_{n,\bs{j}'}(A)$.
Notice that $|n|<|a_{3}|^{2}+\cdots+|a_{l}|^{2}+\frac{l}{2}+1<l(\prod_{k=3}^{l}j_{k}^{*})^{2}$. Then by \eqref{form441}, one has
 \begin{align}
\label{form446}
|\ms{V}_{l}^{(3)}|&\leq\sum_{n,j_{3},\cdots,j_{l}}|\ms{V}'_{n,\bs{j}'}(A')|
\\\notag&\leq\sum_{n,j_{3},\cdots,j_{l}}2(j_{l}^{*})^{m}\frac{2}{A^{m}}
\\\notag&\leq8l\sum_{j_{3},\cdots,j_{l}}(\prod_{k=3}^{l}j_{k}^{*})^{m+2-\frac{\tau m}{3(m+2d)}}\gamma^{\frac{m}{m+2d}}.
\end{align}
Then by \eqref{form446}, the fact $\tau\geq3(m+2d)(m+d+3)/m$ and Lemma \ref{le1} in Appendix, one has
\begin{align}
\label{form447}
|\ms{V}_{l}^{(3)}|
\leq8l(2\times3^{d})^{l-2}\gamma^{\frac{m}{m+2d}}.
\end{align}
\indent To sum up, by estimates \eqref{form442}, \eqref{form443} and \eqref{form447}, one has
\begin{align}
\label{form448}
|\bigcup_{\bs{j}}\ms{V}_{\bs{j}}(\gamma,\tau)|\leq&|\ms{V}_{l}^{(1)}|+|\ms{V}_{l}^{(2)}|+|\ms{V}_{l}^{(3)}|
\\\notag\leq&(8(12l)^{d}+8(4l)^{d}+8l)(2\times3^{d})^{l-2}\gamma^{\frac{m}{m+2d}}
\\\notag\leq&\rho^{l}\gamma^{\frac{m}{m+2d}}
\end{align}
with $\rho=2^{d+4}3^{d}$. In view of $m>d/2$, hence when $\tau\geq15(m+2d+1)$, one has
\begin{equation}
\label{form449}
|\bigcup_{l=3}^{r+2}\bigcup_{\bs{j}\in(\mb{U}_{2}\times\mb{Z}^{d})^{l}\backslash\mc{I}_{l}}\ms{V}_{\bs{j}}(\gamma,\tau)|\leq\sum_{l=3}^{r+2}\rho^{l}\gamma^{\frac{m}{m+2d}}\leq\rho^{r+3}\gamma^{\frac{m}{m+2d}}.
\end{equation}
The proof is completed.
\end{proof}
%%%
\begin{proof}[Proof of Theorem {\sl\ref{th42}}]
By Lemma \ref{le43}, there exists $R'>0$ such that for any $N\geq0$, there exists $C_{N}>0$ such that $R\leq R'$, one has $\langle|P|\rangle_{0,N}^{R}\leq C_{N}R$. This means that the condition \eqref{form211} in Theorem \ref{th21} holds with $s_0=0$.
Assuming $\gamma<5\rho r^{4}C_{N}R'$, then $R_{*}=\frac{\gamma}{10\rho r^{4}C_{N}}$ in \eqref{form213}. Then for any $s\geq r\tau+\frac{d+1}{2}$ and $N=s+1$, one has $\varepsilon_*=\frac{3\gamma}{200\rho^{2}r^4C_{N}}$ in \eqref{form219}. Assuming $\gamma<\rho^{-\frac{(m+2d)(r+3)}{m}}$, then by Lemma \ref{le44} and Corollary \ref{co21}, we can get Theorem \ref{th42} with $c_{1}=\frac{(4\varepsilon_{*})^{r}}{5\rho^{2}C_{N}}$ and $c_{2}=1+\frac{200\rho^2C_{N}}{\gamma}$.
\end{proof}
%%%
Finally, we apply Theorem \ref{th42} to prove \Cref{co42} and \Cref{re41}.
\begin{proof}[Proof of Corollary {\sl\ref{co42}}]
Take $r=[\frac{\log_{\rho}\frac{1}{\varepsilon}}{12\tau\log_{\rho}\log_{\rho}\frac{1}{\varepsilon}}]-3$, $\gamma=\rho^{-\frac{\log_{\rho}\frac{1}{\varepsilon}}{2\tau\log_{\rho}\log_{\rho}\frac{1}{\varepsilon}}}$. Then
$$s=\frac{\log_{\rho}\frac{1}{\varepsilon}}{12\log_{\rho}\log_{\rho}\frac{1}{\varepsilon}}\geq r\tau+\frac{d+1}{2}$$
and for any $m>\frac{d}{2}$, one has
$$|\ms{V}'|\leq\rho^{r+3}\gamma^{\frac{m}{m+2d}}\leq\rho^{-\frac{\log_{\rho}\frac{1}{\varepsilon}}{60\tau\log_{\rho}\log_{\rho}\frac{1}{\varepsilon}}}.$$
By \eqref{form437} with $N=s+1$ in Lemma \ref{le43}, there exist the constants $\mu,C'>0$ such that
\begin{equation}
 \label{form450}
 c_{1}\varepsilon^{-r}=\tilde{c}_{1}(\frac{e\mu}{N+2})^{(N+2)(r+1)}(\frac{\tilde{c}_{0}\gamma }{r^{4}})^{r}\varepsilon^{-r},
\end{equation}
\begin{equation}
\label{form451}
c_{2}\varepsilon^{3}=\big(1+\frac{\tilde{c}_2}{\gamma}(\frac{N+2}{e\mu})^{N+2}\big)\varepsilon^{3}
 \end{equation}
 with $\tilde{c}_{0}=\frac{3}{50\rho^{2}C'}$, $\tilde{c}_1=\frac{1}{5\rho^{2}C'}$ and $\tilde{c}_2=200\rho^{2}C'$.
Notice that for any small enough $\varepsilon$, one has
 \begin{equation}
  \label{form452}
 (N+2)(r+1)\log_{\rho}\frac{N+2}{e\mu}\leq\frac{1}{6}r\log_{\rho}\frac{1}{\varepsilon},
 \end{equation}
 \begin{equation}
 \label{form453}
 \log_{\rho}\frac{r^{4}}{\tilde{c}_0\gamma}\leq\frac{1}{6}\log_{\rho}\frac{1}{\varepsilon}.
 \end{equation}
By \eqref{form450}, \eqref{form452} and \eqref{form453}, one has
\begin{align*}
c_{1}\varepsilon^{-r}&\geq\tilde{c}_{1}\varepsilon^{-\frac{2}{3}r}\geq\varepsilon^{-\frac{\log_{\rho}\frac{1}{\varepsilon}}{24\tau\log_{\rho}\log_{\rho}\frac{1}{\varepsilon}}}=\rho^{\frac{(\log_{\rho}\frac{1}{\varepsilon})^{2}}{24\tau\log_{\rho}\log_{\rho}\frac{1}{\varepsilon}}}.
\end{align*}
Similarly, by \eqref{form451}, one has
\begin{align*}
c_{2}\varepsilon^{3}\leq\varepsilon^{\frac{5}{2}}.
\end{align*}
Hence, we can draw the conclusion.
\end{proof}
%%%

%%
\begin{proof}[Proof of Corollary {\sl\ref{re41}}]
By Lemma \ref{le43} and the definition of norm \eqref{form27}, the estimate \eqref{form436} holds with the constant $C_N$ independent of $N$.
Take $r=[\frac{1}{30}\log_{\rho}\frac{1}{\varepsilon}]-3$, $\gamma=\varepsilon^{\frac{1}{3}}=\rho^{-\frac{1}{3}\log_{\rho}\frac{1}{\varepsilon}}$. Then thus $s\geq\frac{\tau}{9}\log_{\rho}\frac{1}{\varepsilon}\geq r\tau+\frac{d+1}{2}$, $|\ms{V}'|\leq\rho^{r+3}\gamma^{\frac{m}{m+2d}}\leq\varepsilon^{\frac{1}{30}}$,
\begin{align*}
c_{1}\varepsilon^{-r}=&\frac{1}{5\rho^{2}C_N}(\frac{3}{50\rho^{2}C_N}\frac{\gamma}{r^{4}})^{r}\varepsilon^{-r}
\\\geq&\frac{1}{5\rho^{2}C_N}(\frac{3}{50\rho^{2}C_Nr^{4}})^{r}\varepsilon^{-\frac{2}{3}r}
\\\geq&\rho^{\frac{1}{46}(\log_{\rho}\frac{1}{\varepsilon})^{2}}.
\end{align*}
and
$c_{2}\varepsilon^{3}=(1+\frac{200\rho^2C_{N}}{\gamma})\varepsilon^{3}\leq\varepsilon^{\frac{5}{2}}.$
\end{proof}

%=========================================================
\appendix
\setcounter{lemma}{0}
\setcounter{equation}{0}
\renewcommand{\thelemma}{\arabic{lemma}}
\renewcommand{\theequation}{\arabic{equation}}
\section*{Appendix}
%%%%%
\begin{lemma}
\label{le1}
One has
\begin{equation}
\label{form01}
\sum_{j\in\mb{U}_{2}\times\mb{Z}^{d}}\frac{1}{\langle j\rangle^{d+1}}=2\sum_{a\in\mb{Z}^{d}}\frac{1}{\langle a\rangle^{d+1}}<2\times3^{d}.
\end{equation}
\end{lemma}
%%%
\begin{proof}
Let $a=(a_{1},\cdots,a_{d})$ and $|a|_{\infty}=\max\{|a_{1}|,\cdots, |a_{d}|\}$. Then
\begin{align*}
\label{form02}
\sum_{a\in\mb{Z}^{d}}\frac{1}{\langle a\rangle^{d+1}}
&\leq\sum_{a\in\mb{Z}^{d}}\frac{1}{(1+|a|_{\infty})^{d+1}}
\\&=1+\sum_{n\geq1}\frac{(2n+1)^{d}-(2n-1)^{d}}{(1+n)^{d+1}}
\\&\leq1+\sum_{n\geq1}\frac{2d(2n+1)^{d-1}}{(1+n)^{d+1}}
\\&\leq1+2^{d}d\sum_{n\geq1}\frac{1}{(1+n)^{2}}
\\&=1+2^{d}d(\frac{\pi^{2}}{6}-1)
\\&<3^{d}.
\end{align*}
\end{proof}
%%%%%
\begin{lemma}
\label{le2}
For any given $j\in\mb{U}_{2}\times\mb{Z}^{d}$, one has
\begin{equation}
\label{form02}
\sum_{\substack{j_{1},j_{2}\in\mb{U}_{2}\times\mb{Z}^{d}\\\mc{M}(j_{1},j_{2},j)=0}}\frac{\langle j\rangle^{d+1}}{\langle j_{1}\rangle^{d+1}\langle j_{2}\rangle^{d+1}}<2^{d+3}3^{d}.
\end{equation}
\end{lemma}
%%%
\begin{proof}
By $\langle j\rangle\leq\langle j_{1}\rangle+\langle j_{2}\rangle$ and Lemma \ref{le1}, one has
\begin{align*}
\sum_{\substack{j_{1},j_{2}\in\mb{U}_{2}\times\mb{Z}^{d}\\\mc{M}(j_{1},j_{2},j)=0}}\frac{\langle j\rangle^{d+1}}{\langle j_{1}\rangle^{d+1}\langle j_{2}\rangle^{d+1}}
\leq&2^{d}\sum_{\substack{j_{1},j_{2}\in\mb{U}_{2}\times\mb{Z}^{d}\\\mc{M}(j_{1},j_{2},j)=0}}(\frac{1}{\langle j_{1}\rangle^{d+1}}+\frac{1}{\langle j_{2}\rangle^{d+1}})
\\=&2^{d+2}\sum_{j_{1}\in\mb{U}_{2}\times\mb{Z}^{d}}\frac{1}{\langle j_{1}\rangle^{d+1}}
\\<&2^{d+3}3^{d}.
\end{align*}
\end{proof}
%%%%%
\begin{lemma}
\label{le3}
For any $\bs{j}=(\delta_{k},a_{k})_{k=1}^{l}\in(\mb{U}_{2}\times\mb{Z}^{d})^{l}$, one has
\begin{equation}
\label{form03}
j_{1}^{*}\leq\langle\mc{M}(\bs{j})\rangle j_{2}^{*}(\prod_{k=3}^{l}j_{k}^{*})^{\frac{2}{3}}.
\end{equation}
\end{lemma}
%%%
\begin{proof}
By direct calculation, for any real numbers $x\geq1$ and $0\leq y\leq x$, one has
\begin{equation}
\label{form04}
1+x+y\leq(1+x)(1+y)^{\frac{2}{3}}.
\end{equation}
Assuming $|a_{1}|\geq|a_{2}|\geq\cdots\geq|a_{l}|$ and $|a_{1}|\geq1$ without loss of generality, then by induction, one has
\begin{align*}
j_{1}^{*}=&1+|a_{1}|
\\\leq&1+|\mc{M}(\bs{j})|+|a_{2}|+\cdots+|a_{l}|
\\\leq&(1+|\mc{M}(\bs{j})|+|a_{2}|+\cdots+|a_{l-1}|)(1+|a_{l}|)^{\frac{2}{3}}
\\\leq&(1+|\mc{M}(\bs{j})|+|a_{2}|)(1+|a_{3}|)^{\frac{2}{3}}\cdots(1+|a_{l}|)^{\frac{2}{3}}
\\\leq&(1+|\mc{M}(\bs{j})|)(1+|a_{2}|)(1+|a_{3}|)^{\frac{2}{3}}\cdots(1+|a_{l}|)^{\frac{2}{3}}
\\=&\langle\mc{M}(\bs{j})\rangle j_{2}^{*}(\prod_{k=3}^{l}j_{k}^{*})^{\frac{2}{3}}.
\end{align*}
Hence, we can draw the conclusion.
\end{proof}
%%%%%
\begin{lemma}
\label{le4}
For $\omega_{a}=\sqrt{a^{2}+m}$, $a\in\mb{N}$, $m\in[1,2]$, one has the following estimates:
\begin{equation}
\label{form05}
\omega_0^{-2}-\omega_1^{-2}\geq\frac{1}{6}=\frac{\sqrt{2}}{3}\langle 0\rangle^{-\frac{3}{2}}\langle 1\rangle^{-\frac{3}{2}},
\end{equation}
\begin{equation}
\label{form06}
\omega_0^{-2}-\omega_a^{-2}\geq\frac{1}{3}>\langle 0\rangle^{-\frac{3}{2}}\langle a\rangle^{-\frac{3}{2}}, \quad\text{for any}\; a\geq2,
\end{equation}
\begin{equation}
\label{form07}
\omega_b^{-2}-\omega_a^{-2}>\langle b\rangle^{-3}>\langle b\rangle^{-\frac{3}{2}}\langle a\rangle^{-\frac{3}{2}}, \quad\text{for any}\; a>b\geq1.
\end{equation}
\end{lemma}
%%%
\begin{proof}
These estimates follow from direct calculation.
\end{proof}
%%%%%
\begin{lemma}
\label{le5}
Let $u^{(1)},\dots,u^{(K)}$ be $K$ independent vectors with $\|u^{(i)}\|_{\ell^{1}}\leq1$. Let $w\in\mb{R}^{K}$ be an arbitrary vector, then there exists $i_{0}\in\{1,\dots,K\}$ such that
\begin{equation}
\label{form08}
|w\cdot u^{(i_{0})}|{\color{blue}\geq}\frac{\|w\|_{\ell^{1}}|\det(u)|}{K^{\frac{3}{2}}},
\end{equation}
where $\|w\|_{\ell^{1}}=\sum_{k=1}^{K}|w_{k}|$ and $\det(u)$ is the determinant of the matrix formed by the components of the vectors $u^{(1)},\cdots,u^{(K)}$.
\end{lemma}
%%%
\begin{proof}
It is Lemma 6.9 of \cite{B03}.
\end{proof}
%%%%%
\begin{lemma}
\label{le6}
Let $h$ is $l+1$ times differentiable on an interval $\mc{P}\subseteq\mb{R}$ and $\mc{P}_{X}=\left\{x\in\mc{P}\mid|h(x)|<X\right\}$. Assume that
\\(i) for any $x\in\mc{P}$, there exists $i_{0}\leq l$ such that $h^{(i_{0})}(x)>A$,
\\(ii) there exists $B$ such that for any $x\in\mc{P}$ and $1\leq i\leq l+1$, one has $|h^{(i)}(x)|\leq B$. Then
\begin{equation}
\label{form09}
\frac{|\mc{P}_{X}|}{|\mc{P}|}\leq\frac{B}{A}2(2+3+\cdots+l+\frac{2}{A})X^{\frac{1}{l}}.
\end{equation}
\end{lemma}
%%%
\begin{proof}
It is Lemma 8.4 of \cite{B99}.
\end{proof}
%%%%%
%With the similar argument in the subsection 6.2 of \cite{B03}, we give the measure estimate of the resonant set $\ms{P}_{\bs{j}}(\gamma,\tau)$ defined in \eqref{form414}.
%%%%%
\begin{lemma}
\label{le7}
Assume $\gamma<1$. Then for any $\bs{j}\in(\mb{U}_{2}\times\mb{Z})^{l}\backslash\mc{J}_{l}$, one has
\begin{equation}
\label{form010}
|\ms{P}_{\bs{j}}(\gamma,\tau)|\leq2^{3l+6}l^{9}(j_{2}^{*})^{4}(j_{3}^{*}\cdots j_{l}^{*})^{3l+2-\frac{\tau}{3l}}\gamma^{\frac{1}{l}},
\end{equation}
where $\ms{P}_{\bs{j}}(\gamma,\tau)$ is defined in \eqref{form414}.
\end{lemma}
%%%
\begin{proof}
The process is similar to Lemma 6.8 and Lemma 6.12 in \cite{B03}. However, for the purpose to obtain concrete constants on the right side of  \eqref{form010}, we need give the details.
At first, for any integer $3\leq K\leq l$, consider indexes $0\leq a'_{1}<\cdots<a'_{K}$ and the determinant
$$D:=\begin{vmatrix}
\omega_{a'_{1}} & \omega_{a'_{2}} & \cdots & \omega_{a'_{K}}\\
\frac{d\omega_{a'_{1}}}{dm} & \frac{d\omega_{a'_{2}}}{dm} &  \cdots & \frac{d\omega_{a'_{K}}}{dm}\\
\vdots & \vdots &   & \vdots\\
\frac{d^{K-1}\omega_{a'_{1}}}{dm^{K-1}} & \frac{d^{K-1}\omega_{a'_{2}}}{dm^{K-1}} &  \cdots &
\frac{d^{K-1}\omega_{a'_{K}}}{dm^{K-1}}
\end{vmatrix}.$$
Notice that
\begin{equation}
\label{form011}
\frac{d\omega_{a}}{dm}=\frac{1}{2}\omega_{a}^{-1}\quad\text{and}\quad
\frac{d^{i}\omega_{a}}{dm^{i}}=\frac{(2i-3)!}{(-4)^{i-1}(i-2)!}\omega_{a}^{-2i+1}\;\text{for any}\;i\geq2.
\end{equation}
Hence, one has
\begin{equation}
\label{form012}
D=\frac{1}{2}\left(\prod_{k=2}^{K-1}\frac{(2k-3)!}{(-4)^{k-1}(k-2)!}\right)\left(\prod_{k=1}^{K}\omega_{a'_{k}}\right)
\begin{vmatrix}
1 & 1 & \cdots & 1\\
\omega_{a'_{1}}^{-2} & \omega_{a'_{2}}^{-2} & \cdots & \omega_{a'_{K}}^{-2}\\
\vdots & \vdots &   & \vdots\\
\omega_{a'_{1}}^{-2(K-1)} & \omega_{a'_{2}}^{-2(K-1)} & \cdots & \omega_{a'_{K}}^{-2(K-1)}
\end{vmatrix},
\end{equation}
where the last Vandermonde determinant about $\omega_{n'_{k}}^{-2}\ (k=1,\cdots,K)$ is equal to
\begin{equation}
\label{form013}
\prod_{1\leq i<k\leq K}(\omega_{a'_{k}}^{-2}-\omega_{a'_{i}}^{-2}).
\end{equation}
By Lemma \ref{le4}, one has
\begin{eqnarray}
\label{form014}
\prod_{1\leq i<k\leq K-2}(\omega_{a'_{i}}^{-2}-\omega_{a'_{k}}^{-2}) & \geq & \frac{\sqrt{2}}{3}\prod_{1\leq i\leq K-2}\langle a'_{i}\rangle^{-\frac{3}{2}(K-3)},
\\\label{form015}\prod_{1\leq i\leq K-2}(\omega_{a'_{i}}^{-2}-\omega_{a'_{K-1}}^{-2}) & \geq &  \frac{1}{6}\prod_{1\leq i\leq K-2}\langle a'_{i}\rangle^{-3},
\\\label{form016}\prod_{1\leq i\leq K-2}(\omega_{a'_{i}}^{-2}-\omega_{a'_{K}}^{-2}) & \geq &\frac{1}{3}\prod_{1\leq i\leq K-2}\langle a'_{i}\rangle^{-3},
\\\label{form017}\omega_{a'_{K-1}}^{-2}-\omega_{a'_{K}}^{2} & \geq &\langle a'_{K-1}\rangle^{-3}.
\end{eqnarray}
By \eqref{form012}--\eqref{form017}, one has
\begin{equation}
\label{form018}
|D|\geq\frac{1}{54\sqrt{2}}\left(\prod_{k=2}^{K-1}\frac{(2k-3)!}{4^{k-1}(k-2)!}\right)\left(\prod_{k=1}^{K}\omega_{a'_{k}}\right)\frac{1}{(\prod_{1\leq i\leq K-2}\langle a'_{i}\rangle^{\frac{3}{2}(K+1)})\langle a'_{K-1}\rangle^{3}}.
\end{equation}
Next, we estimate the sum of elements of each row of $D$. By \eqref{form011}, one has
\begin{align}
\label{form019}&\sum_{k=1}^{K}\omega_{a'_{k}}\leq K\omega_{a'_{K}},
\\\label{form020}&\Big|\frac{d}{dm}(\sum_{k=1}^{K}\omega_{a'_{k}})\Big|=\frac{1}{2}\sum_{k=1}^{K}\omega_{a'_{k}}^{-1}\leq\frac{K}{2},
\\\label{form021}&\Big|\frac{d^{i}}{dm^{i}}(\sum_{k=1}^{K}\omega_{a'_{k}})\Big|=\frac{(2i-3)!}{4^{i-1}(i-2)!}\left(\sum_{k=1}^{K}\omega_{a'_{k}}^{1-2i}\right)\leq\frac{2(2i-3)!}{4^{i-1}(i-2)!},\;\text{for any}\;i\geq2.
\end{align}
Hence, one has
\begin{equation}
\label{form022}
\prod_{i=0}^{K-1}\Big|\frac{d^{i}}{dm^{i}}(\sum_{k=1}^{K}\omega_{a'_{k}})\Big|\leq2^{K-3}K^{2}\left(\prod_{k=2}^{K-1}\frac{(2k-3)!}{4^{k-1}(k-2)!}\right)\omega_{a'_{K}}.
\end{equation}
By Lemma \ref{le5},  \eqref{form018}, \eqref{form022} and the fact $\omega_{a_{k}}\geq\langle a_{k}\rangle/\sqrt2$, there exists $i_{0}\in\{0,\cdots,K-1\}$ such that
\begin{align}
\label{form023}
\Big|w\cdot\frac{d^{i_{0}}\omega}{dm^{i_{0}}}\Big|&\geq\frac{\|w\|_{\ell^{1}}|D|}{K^{\frac{3}{2}}\prod_{i=0}^{K-1}\big|\frac{d^{i}}{dm^{i}}(\sum_{k=1}^{K}\omega_{a'_{k}})\big|}
\\\notag&\geq\frac{\|w\|_{\ell^{1}}}{2^{\frac{3}{2}K+3}K^{\frac{7}{2}}(\prod_{k=1}^{K-2}\langle a'_{k}\rangle)^{\frac{3}{2}K+\frac{1}{2}}\langle a'_{K-1}\rangle^{2}},
\end{align}
where $\omega=(\omega_{a'_{1}},\cdots,\omega_{a'_{K}})$ and $w\in\mb{R}^{K}$ is a vector whose components are not zero. Moreover, for $K=2$, \eqref{form023} still holds with convention that $\prod_{k=1}^{K-2}\langle a'_{k}\rangle=1$.
\\\indent Finally,  for any $\bs{j}=(\delta_{k},a_{k})_{k=1}^{l}\in(\mb{U}_{2}\times\mb{Z})^{l}\backslash\mc{J}_{l}$, we estimate the measure of $\ms{P}_{\bs{j}}(\gamma,\tau)$. Let $h(m)=\sum_{k=1}^{l}\delta_{k}\omega_{a_{k}}(m)$. Then $h(m)$ could be expressed by $\sum_{k=1}^{K}w_{k}\omega_{a'_{k}}$ with $w_{k}\neq0$ and $0\leq a'_{1}<\cdots<a'_{K}$. If $K=1$, then $\ms{P}_{\bs{j}}(\gamma,\tau)=\emptyset$. So we only need to consider the case $2\leq K\leq l$, and thus $2\leq\|w\|_{\ell^{1}}\leq l$. By \eqref{form023}, for any $m\in[1,2]$, there exists $i_{0}\in\{0,\cdots,K-1\}$ such that
\begin{equation}
\label{form024}
|h^{(i_{0})}(m)|\geq\frac{1}{2^{\frac{3}{2}l+2}l^{\frac{7}{2}}(\prod_{k=3}^{l}j_{k}^{*})^{\frac{3}{2}l+\frac{1}{2}}(j_{2}^{*})^{2}}.
\end{equation}
For any $i=1,\cdots,K$, in view of \eqref{form011}, one has
$$|h^{(i)}(m)|\leq\|w\|_{\ell^{1}}\sup_{1\leq k\leq K}\Big|\frac{d^{i}\omega_{a'_{k}}}{dm^{i}}\Big|\leq l\frac{(i-1)!}{2}\leq l\frac{(K-1)!}{2}.$$
Notice that $\prod_{k=1}^{K-2}\langle a'_{k}\rangle\geq(K-2)!$.
Hence, one has
\begin{equation}
\label{form025}
|h^{(i)}(m)|\leq\frac{l(K-1)}{2}\prod_{k=1}^{K-2}\langle a'_{k}\rangle\leq\frac{l^{2}}{2}\prod_{k=3}^{l}j_{k}^{*}.
\end{equation}
By \eqref{form024}, \eqref{form025}, Lemma \ref{le6} and Lemma \ref{le3}, one has
\begin{align*}
|\ms{P}_{\bs{j}}(\gamma,\tau)|\leq&8\Big(\frac{l^{2}}{2}\prod_{k=3}^{l}j_{k}^{*}\Big)\Big(2^{\frac{3}{2}l+2}l^{\frac{7}{2}}(\prod_{k=3}^{l}j_{k}^{*})^{\frac{3}{2}l+\frac{1}{2}}(j_{2}^{*})^{2}\Big)^{2}\Big(\gamma(\frac{j_{1}^{*}}{\langle\mc{M}(\bs{j})\rangle j_{2}^{*}\dots j_{l}^{*}})^{\tau}\Big)^{\frac{1}{K-1}}
\\\leq&2^{3l+6}l^{9}(j_{2}^{*})^{4}(\prod_{k=3}^{l}j_{k}^{*})^{3l+2}\Big(\gamma(\frac{j_{1}^{*}}{\langle\mc{M}(\bs{j})\rangle j_{2}^{*}\dots j_{l}^{*}})^{\tau}\Big)^{\frac{1}{l}}
\\\leq&2^{3l+6}l^{9}(j_{2}^{*})^{4}(\prod_{k=3}^{l}j_{k}^{*})^{3l+2-\frac{\tau}{3l}}\gamma^{\frac{1}{l}}.
\end{align*}
The proof is completed.
\end{proof}
%%%%%
%assume $j_{1}^{*}=\langle a_{1}\rangle, j_{2}^{*}=\langle a_{2}\rangle$ without loss of generality.
For any $\bs{j}=(\delta_{k},a_{k})_{k=1}^{l}\in(\mb{U}_{2}\times\mb{Z})^{l}\backslash\mc{J}_{l}$, if $\delta_{1}\delta_{2}=1$ or $\delta_{1}\delta_{2}=-1$ and $j_{2}^{*}$ is not too big, then we estimate the measure of $\ms{P}_{\bs{j}}(\gamma,\tau)$ by Lemma \ref{le7}. Otherwise, we estimate $\ms{P}_{\bs{j}}(\gamma,\tau)$ by
\begin{equation}
\label{form026}
\ms{P}'_{n,\bs{j}'}(X):=\big\{m\in[1,2] \mid |n+\delta_{3}\omega_{a_{3}}+\dots+\delta_{l}\omega_{a_{l}}|<X\big\}
\end{equation}
with $n:=\delta_{1}(|a_{1}|-|a_{2}|)$, $\bs{j}':=(\delta_{k},a_{k})_{k=3}^{l}$ and an appropriate $X$.
%%%%%
\begin{lemma}
\label{le8}
For the above $n,\bs{j}'$, one has
\begin{equation}
\label{form027}
|\ms{P}'_{n,\bs{j}'}(X)|\leq2^{3l-3}l^{6}(j_{3}^{*}\cdots j_{l}^{*})^{3l-6}X^{\frac{1}{l}}.
\end{equation}
\end{lemma}
%%%
\begin{proof}
Similarly with the proof of Lemma \ref{le7}, for any integer $K\geq2$, consider indexes $0\leq a'_{1}<\cdots<a'_{K}$ and the determinant
$$D':=\begin{vmatrix}
\frac{d\omega_{a'_{1}}}{dm} & \frac{d\omega_{a'_{2}}}{dm} &  \cdots & \frac{d\omega_{a'_{K}}}{dm}\\
\vdots & \vdots &   & \vdots\\
\frac{d^{K}\omega_{a'_{1}}}{dm^{K}} & \frac{d^{K}\omega_{a'_{2}}}{dm^{K}} &  \cdots &
\frac{d^{K}\omega_{a'_{K}}}{dm^{K}}
\end{vmatrix}.$$
By \eqref{form011}, \eqref{form013}, Lemma \ref{le4} and $\omega_{a}\leq\sqrt{2}\langle a\rangle$, one has
\begin{align*}
|D'|&=\frac{1}{2}\left(\prod_{k=2}^{K}\frac{(2k-3)!}{4^{k-1}(k-2)!}\right)\left(\prod_{k=1}^{K}\omega_{a'_{k}}^{-1}\right)\prod_{1\leq i<k\leq K}|\omega_{a'_{k}}^{-2}-\omega_{a'_{i}}^{-2}|
\\&\geq\frac{1}{6(\sqrt{2})^{K-1}}\left(\prod_{k=2}^{K}\frac{(2k-3)!}{4^{k-1}(k-2)!}\right)\prod_{k=1}^{K}\langle a'_{k}\rangle^{\frac{1}{2}-\frac{3}{2}K}.
\end{align*}
Notice that by \eqref{form020} and \eqref{form021}, one has
$$\prod_{i=1}^{K}|\frac{d^{i}}{dm^{i}}(\sum_{k=1}^{K}\omega_{a'_{k}})|\leq2^{K-2}K\left(\prod_{i=2}^{K}\frac{(2i-3)!}{4^{i-1}(i-2)!}\right).$$
Hence, by Lemma \ref{le5}, for any $m\in[1,2]$, there exists $i_{0}\in\{1,\cdots,K\}$ such that
\begin{equation}
\label{form028}
|w\cdot\frac{d^{i_{0}}\omega}{dm^{i_{0}}}|\geq\frac{\|w\|_{\ell^{1}}|D'|}{K^{\frac{3}{2}}\prod_{i=1}^{K}|\frac{d^{i}}{dm^{i}}(\sum_{k=1}^{K}\omega_{a'_{k}})|}
\geq\frac{\|w\|_{\ell^{1}}}{2^{\frac{3}{2}K+\frac{1}{2}}K^{\frac{5}{2}}\prod_{k=1}^{K}\langle a'_{k}\rangle^{\frac{3}{2}K-\frac{1}{2}}},
\end{equation}
where $\omega=(\omega_{a'_{1}},\cdots,\omega_{a'_{K}})$ and $w\in\mb{R}^{K}$ is a vector whose components are not zero. Moreover, for $K=1$, \eqref{form028} still holds.
\\\indent Finally, we estimate the measure of $\ms{P}'_{n,\bs{j}'}(X)$. Let $h(m)=n+\sum_{k=3}^{l}\delta_{k}\omega_{a_{k}}(m)$. Then $h(m)$ could be expressed by $n+\sum_{k=1}^{K}w_{k}\omega_{a'_{k}}$ with $w_{k}\neq0$ and $0\leq a'_{1}<\cdots<a'_{K}$. Obviously, $1\leq K\leq l-2$ and then by \eqref{form028}, for any $m\in[1,2]$, there exists $i_{0}\in\{1,\cdots,K\}$ such that
\begin{equation}
\label{form029}
|h^{(i_{0})}(m)|>\frac{1}{2^{\frac{3}{2}(l-2)+\frac{1}{2}}l^{\frac{5}{2}}(\prod_{k=3}^{l}j_{k}^{*})^{\frac{3}{2}l-\frac{7}{2}}}.
\end{equation}
In view of \eqref{form011}, for any $i=1,\cdots,K+1$, one has
$$|h^{(i)}(m)|\leq l\frac{(i-1)!}{2}\leq l\frac{K!}{2}.$$
Notice that $\prod_{k=1}^{K}\langle a'_{k}\rangle\geq K!$ and thus
\begin{equation}
\label{form030}
|h^{(i)}(m)|\leq \frac{l}{2}\prod_{k=1}^{K}\langle a'_{k}\rangle\leq\frac{l}{2}\prod_{k=3}^{l}j_{k}^{*}.
\end{equation}
By \eqref{form029}, \eqref{form030} and Lemma \ref{le6}, one has
\begin{align*}
|\ms{P}'_{n,\bs{j}'}(\gamma,\tau)|&\leq8(\frac{l}{2}\prod_{k=3}^{l}j_{k}^{*})\big(2^{\frac{3}{2}(l-2)+\frac{1}{2}}l^{\frac{5}{2}}(\prod_{k=3}^{l}j_{k}^{*})^{\frac{3}{2}l-\frac{7}{2}}\big)^{2}X^{\frac{1}{K}}
\\&\leq2^{3l-3}l^{6}(\prod_{k=3}^{l}j_{k}^{*})^{3l-6}X^{\frac{1}{l}}.
\end{align*}
The proof is completed.
\end{proof}
%%%%%

%==========================================
%\bibliographystyle{alpha}%abbrv
%\bibliography{reference}

%===============================

\end{document}